\documentclass[journal,10pt,onecolumn]{IEEEtran}
\usepackage{amsmath,amssymb,dsfont,graphicx,color,lmodern}
\usepackage{amsthm}
\usepackage[T1]{fontenc}
\usepackage[utf8]{inputenc}

\begin{document}
\newtheorem{thm}{Theorem}
\newtheorem{cor}[thm]{Corollary}
\newtheorem{conj}[thm]{Conjecture}
\newtheorem{lemma}[thm]{Lemma}
\newtheorem{proposition}[thm]{Proposition}
\newtheorem{problem}{Problem}
\newtheorem{remark}{Remark}
\newtheorem{definition}{Definition}
\newtheorem{example}{Example}

\newcommand{\argmax}{\mathop{\operatorname{argmax}}}
\newcommand{\bp}{{\bm{p}}}
\newcommand{\bq}{{\bm{q}}}
\newcommand{\bc}{{\bm{c}}}
\newcommand{\be}{{\bm{e}}}
\newcommand{\br}{{\bm{r}}}
\newcommand{\bv}{{\bm{v}}}
\newcommand{\bx}{{\bm{x}}}
\newcommand{\bbf}{{\bm{f}}}
\newcommand{\balpha}{{\bm{\alpha}}}
\newcommand{\bbeta}{{\bm{\beta}}}
\newcommand{\bdelta}{{\bm{\delta}}}
\newcommand{\bPi}{{\bm{\Pi}}}
\newcommand{\bC}{{\bm{C}}}
\newcommand{\bL}{{\bm{L}}}
\newcommand{\bI}{{\bm{I}}}
\newcommand{\bP}{{\bm{P}}}
\newcommand{\bD}{{\bm{D}}}
\newcommand{\bQ}{{\bm{Q}}}
\renewcommand{\dagger}{{T}}
\newcommand{\cE}{{\mathcal{E}}}
\newcommand{\cS}{{\mathcal{S}}}
\newcommand{\cF}{{\mathcal{F}}}
\newcommand{\cW}{{\mathcal{W}}}
\newcommand{\cQ}{{\mathcal{Q}}}
\newcommand{\differential}{{\rm{d}}}
\newcommand{\interior}[1]{{\kern0pt#1}^{\mathrm{o}}}
\newcommand{\mR}{{\mathbb R}}
\newcommand{\mZ}{{\mathbb Z}}
\newcommand{\Prob}{{\rm Pr}}
\newcommand{\D}{{\mathbb D}}
\newcommand{\cX}{{\mathcal X}}
\newcommand{\cT}{{\mathcal T}}
\newcommand{\E}{{\mathbb E}}
\newcommand{\mcN}{{\mathcal N}}
\newcommand{\mcR}{{\mathcal R}}
\newcommand{\diag}{\operatorname{diag}}
\newcommand{\tr}{\operatorname{trace}}
\renewcommand{\odot}{\circ}
\newcommand{\ignore}[1]{}
\newcommand{\ud}{{\rm{d}}}
\newcommand{\W}{{W}}
\newcommand{\tf}{t_{\rm cycle}}
\newcommand{\Wirr}{W_\text{irr}}
\newcommand{\Pac}{\mathcal{P}_{2,\text{ac}}(\mR^d)}
\newcommand{\calA}{\mathcal{A}}
\newcommand{\calS}{\mathcal{S}}
\newcommand{\Ltwo}{{\rho}}
\newcommand{\kbt}{k_BT}
\newcommand{\dbar}{d\hspace*{-0.08em}\bar{}\hspace*{0.1em}}
\newcommand{\lr}[2]{\langle#1,#2\rangle}
\newcounter{rmnum}
\newenvironment{romannum}{\begin{list}{{\upshape (\roman{rmnum})}}{\usecounter{rmnum}
			\setlength{\leftmargin}{12pt}
			\setlength{\rightmargin}{8pt}
			\setlength{\itemsep}{2pt}
			\setlength{\itemindent}{-1pt}
	}}{\end{list}}

\newcommand{\magenta}{\color{magenta}}
\newcommand{\red}{\color{red}}
\newcommand{\blue}{\color{blue}}
\definecolor{gray}{rgb}{0.5,0.5,0.5}
\newcommand{\gray}{\color{gray}}

\definecolor{lgrey}{rgb}{0.9,.7,0.7}

\newcommand{\symsum}{\displaystyle\sum_{\rm{symm}}}
\newcommand{\Amir}[1]{{\color{red}#1}}
\def\spacingset#1{\def\baselinestretch{#1}\small\normalsize}
\setlength{\parindent}{10pt}
\setlength{\parskip}{10pt}
\spacingset{1}

\title{
Maximal power output of\\a stochastic thermodynamic engine}

\author{Rui Fu, \and Amirhossein Taghvaei, \and Yongxin Chen, \and Tryphon T.\ Georgiou
\thanks{R.\ Fu, A.\ Taghvaei, and T.\ T. Georgiou are with the Department of Mechanical and Aerospace Engineering, University of California, Irvine, CA; rfu2@uci.edu, ataghvae@uci.edu, tryphon@uci.edu}
\thanks{Y.~Chen is with the School of Aerospace Engineering, Georgia Institute of Technology, Atlanta, GA 30332; {yongchen@gatech.edu}}}

\markboth{\today}{ }

\maketitle

\begin{abstract}
Classical thermodynamics aimed to quantify the efficiency of thermodynamic engines, by bounding the maximal amount of mechanical energy produced, compared to the amount of heat required. While this was accomplished early on, by Carnot and Clausius, the more practical problem to quantify limits of power that can be delivered, remained elusive due to the fact that quasistatic processes require infinitely slow cycling, resulting in a vanishing power output. Recent insights, drawn from stochastic models, appear to bridge the gap between theory and practice in that they lead to physically meaningful expressions for the dissipation cost in operating a thermodynamic engine over a finite time window. 
Indeed, the problem to optimize power can be expressed as a stochastic control problem.
Building on this framework of {\em stochastic thermodynamics} we derive bounds on the maximal power that can be drawn by cycling an overdamped ensemble of particles via a time-varying potential while alternating contact with heat baths of different temperature ($T_c$ cold, and $T_h$ hot). Specifically, assuming a suitable bound $M$ on the spatial gradient of the controlling potential, we show that the maximal achievable power is bounded by $\frac{M}{8}(\frac{T_h}{T_c}-1)$. Moreover, we show that this bound can be reached to within a factor of $(\frac{T_h}{T_c}-1)/(\frac{T_h}{T_c}+1)$ by operating the cyclic thermodynamic process with a quadratic potential.

 \end{abstract}

\noindent{\bf Keywords:} Non-equilibrium Thermodynamics, Optimal Transportation, Mean-field optimal control


\section{Introduction}
Thermodynamics is the branch of physics which is concerned with the relation between heat and other forms of energy. Historically, it was born of the quest to quantify the maximal efficiency of heat engines, i.e., the maximal ratio of the total work output over the total heat input to a thermodynamic system. This was accomplished in the celebrated work of Carnot ~\cite{carnot1986reflexions,callen1998thermodynamics} where, assuming that transitions take place infinitely slowly ({\it quasi-static} operation), it was shown that the maximal efficiency possible is $\eta_C=1-T_c/T_h$ ({\em Carnot efficiency}), where $T_h$ and $T_c$ are the absolute temperatures of two heat reservoirs, hot and cold respectively, with which the heat engine makes contact with during phases of a periodic operation known as {\em Carnot cycle}. 

Carnot's result provides the absolute theoretical limit for the efficiency of a heat engine, but provides no insight on the amount of power output that can be achieved. Specifically, in order to reach Carnot efficiency, the period of the Carnot cycle must tend to infinity, resulting in quasi-static operation with vanishing total power output.
Whereas, to achieve finite power output in a thermodynamic process, this must take place in finite time, and thereby, away from {\em equilibrium} \cite{casas2003temperature,lebon2008understanding,de2013non}.

To this end, the framework of stochastic thermodynamics~\cite{seifert2008stochastic,sekimoto2010stochastic,seifert2012stochastic,parrondo2015thermodynamics,dechant2016underdamped,brockett2017thermodynamics} has been developed in recent years, to allow quantifying work in non-equilibrium thermodynamic transitions. It is rooted in probabilistic models in the form of stochastic differential equations to specify the behaviour of particles in a thermodynamic ensemble. Manipulation of the ensemble is effected by a confining potential that serves as a {\em control input}. This potential, together with a heat reservoir in contact, couples the ensemble to the environment. Work and heat being transferred can then be computed at the level of individual particles and averaged over the ensemble.
Important goals of the theory have been to assess the amount of work needed for {\em bit-erasure in finite time} \cite{TalBhaSal17,MelTalSal18} and hence computation, i.e., a finite-time Landauer bound, as well as assessing the efficiency of thermodynamic engines operating at maximal power.

The question of efficiency at maximal power was studied independently by Chambadal \cite{chambadal1957centrales}, Novikov \cite{novikov1958efficiency} and Curzon and Ahlborn   \cite{curzon1975efficiency} based on a certain ``endoreversible'' assumption to reflect finite-time heat transfer. They derived the bound $\eta_{CA}=1-\sqrt{T_c/T_h}=1-\sqrt{1-\eta_C}$, where the $T_h$ and $T_c$ designate temperatures of a hot and cold heat reservoir, respectively, at maximal power estimated to be $k(\sqrt {T_h}-\sqrt{T_c})^2$, with $k$ being the heat conductance. Subsequent works, most notably by Chen and Yan \cite{chen1989effect}, based on differing sets of assumptions, arrived at different bounds. More recently Schmiedl and Seifert \cite{schmiedl2007efficiency}, sought to improve, and reconcile these earlier results within the framework of stochastic thermodynamics, albeit for thermodynamic ensembles transitioning between Gaussian distributions. It is fair to say that there is no consensus on the firmness of these expressions, and that they serve as a guide to actual performance of thermodynamic engines. 

The present work focuses on maximizing power in general, relaxing the Gaussian assumption, within the context of stochastic thermodynamics \cite{sekimoto2010stochastic,seifert2012stochastic}. This is a {\em stochastic control problem}. Our analysis is based on an overdamped Langevin model for thermodynamic processes (with damping coefficient $\gamma$), and explores advantages and pitfalls of selecting arbitrary control input, i.e., confining potential, for steering thermodynamic ensembles through cyclic operation while alternating contact between available heat reservoirs. It is noted that without physically motivated constraints on the actuation potential, the power ouput can become unbounded. The salient feature of actuation (time-varying potential $U(t,x)$, with $t$ denoting time and $x\in\mR^d$ the spacial coordinate) that draws increasing amounts of power is its ability to drive the thermodynamic ensemble to a state of very low entropy.
Indeed, the magnitude of the spatial gradient of the potential $\nabla_xU(t,x)$ plays a key role. Thus, it is reasonable on physical grounds to suitably constrain this mode of ``control'' actuation, that is responsible for energy exchange between the ensemble and the environment. The present work puts forth and motivates the bound\footnote{Interestingly, this can also be expressed in information theoretic terms, as a bound on the Fisher information of thermodynamic states.} (equation \eqref{eq:U-constraint})
\[
\frac{1}{\gamma}\int_{\mR^d}\|\nabla_x U(t,x)\|^2 \rho(t,x)\,\ud x \leq M,
\]
where $\rho$ denotes the thermodynamic state, as a suitable such constraint, and under this assumption it is shown that a maximal amount of power output that can be extracted by cyclic operation of a Carnot-like engine is
\[
 \frac{M}{8}(\frac{T_h}{T_c}-1) \left(\frac{\frac{T_h}{T_c}-1}{\frac{T_h}{T_c}+1}\right)\leq P_{\rm max}\leq  \frac{M}{8}(\frac{T_h}{T_c}-1).
\]
That is, the upper bound $\frac{M}{8}(\frac{T_h}{T_c}-1)$ on power output only depends on $M$ and the temperature of the two heat baths\footnote{In general power output is an {\em extensive} quantity, as it depends on the size of the thermodynamic ensemble/engine. However, in our treatment, the ensemble is described by a probability distribution (normalized). Hence, the bounds appear as ``intensive.''}. Moreover, this bound can be attained within a factor of $(\frac{T_h}{T_c}-1)/(\frac{T_h}{T_c}+1)$, which depends only on the ratio of temperatures of the two heat baths as well.

The exposition proceeds as follows. Section \ref{sec:model} details the stochastic model thermodynamic ensembles and the heat/energy exchange mechanism. Section \ref{Sec-III-Second-law.tex} explores a connection between the second law of thermodynamics and the Wasserstein geometry of optimal mass transport that underlies the mechanism of energy dissipation in thermodynamic transitions. Section \ref{Sec-IV-cycle.tex} returns to the concept of a cyclicly operated thermodynamic engine and expresses the optimal efficiency and power output as functions of the operating protocol (solution of a stochastic control problem that dictates the choice of control time-varying potential), temperature of heat reservoirs, timing of the cyclic operation, and thermodynamic states at the end of phases of the Carnot-like cycle. Section \ref{Sec-V-power.tex} contains the main results regarding seeking maximal power output. Specifically, Section \ref{sec:51} explains optimal scheduling times, Section \ref{sec:caveat} highlights questions that arise based on physical grounds for Gaussian thermodynamic states, Sections \ref{sectionVC} and \ref{sec:opt-pa} discuss optimal thermodynamic states at the two ends of the Carnot-like cycle, and Sections \ref{sec:55} and \ref{sec:56} derive bounds on maximal achievable power with or without constraint on the controlling potential.
A concluding remarks section recaps and points to future research directions and open problems.

\section{Stochastic thermodynamic models}
\label{sec:model}

We begin by describing the basic model for a {\em thermodynamic ensemble} used in this work. This consists of a large collection of Brownian particles that interact with a {\em heat bath} in the form of a stochastic excitation and driven under the influence of an {\em external (time varying) potential} between end-point states. The dynamics of individual particles are expressed in the form of stochastic differential equations. 

\subsection{Langevin dynamics}

The (under-damped) Langevin equations
\begin{subequations}\label{eq:underdamped-Langevin}
\begin{align}
\ud X_t &= \frac{p_t}{m} \ud t\label{eq:underdamped-Langevin-x}\\
\ud p_t &= -\nabla_x U (t,X_t)\ud t - \gamma \frac{p_t}{m} \ud t + \sqrt{2\gamma k_BT(t)}\ud B_t,\label{eq:underdamped-Langevin-p}
\end{align}
\end{subequations}
represent a standard model for molecular systems interacting with a thermal environment. 
Throughout, 
$X_t\in \mR^d$ denotes the location of a particle and $p_t$ denotes its momentum at time $t$, $U(t,x)$ denotes a time-varying potential for $x\in \mR^d$, $m$ is the mass of the particle, $\gamma$ is the viscosity coefficient, $k_B$ is the Boltzmann constant, $T(t)$ denotes the temperature of the heat bath at time $t$, and $B_t$ denotes a standard $\mR^d$-valued Brownian motion. 

In this paper, we consider only the case where inertial effects in the Langevin equation~\eqref{eq:underdamped-Langevin-p} are negligible for the time resolution of interest. 
Specifically, when the temporal resolution $\Delta t \gg \frac{m}{\gamma}$, averaging out the fast variable $p_t$ leads to the {\em over-damped Langevin equation}
\begin{equation}
\ud X_t = - \frac{1}{\gamma} \nabla_x U (t,X_t)\ud t + \sqrt{\frac{2k_BT(t)}{\gamma}}\ud B_t.
\label{eq:overdamped-Langevin}
\end{equation} 
Intuitively, the over-damped Langevin equation is obtained from~\eqref{eq:underdamped-Langevin-p} by setting $\ud p_t = 0$ and replacing $\frac{p_t}{m}\ud t = \ud X_t$. For a more detailed explanation see~\cite[page 20]{sekimoto2010stochastic}.

Thus, we view $\{X_t\}_{t \geq 0}$ as a diffusion process. The state of the thermodynamic ensemble is identified with the probability density of $X_t$, denoted by $\rho(t,x)$, which satisfies
the Fokker-Planck equation
\begin{equation}
\label{eq: Fokker-Planck}
\frac{ \partial{\rho}}{\partial{t}}-\frac{1}{\gamma}\nabla_x\cdot \left[ (\nabla_x U +k_BT\nabla_x \log \rho)\rho\,\right]=0.
\end{equation}

\subsection{Heat, work, and the first law}

The evolution of the thermodynamic ensemble under the influence of the time-varying thermal environment and the 
time-varying potential $U(t,x)$, leads to exchange of heat and work, respectively. Heat and work can be defined at the level of a single particle as explained below. 

The energy exchange between an individual particle and the thermal environment
represents {\em heat}. This exchange is effected by forces exerted on the particle due to viscosity ($-\gamma \frac{\ud X_t}{\ud t}$) and due to the random thermal excitation ($\sqrt{2\gamma k_BT}\frac{\ud B_t}{\ud t}$). It can be formally expressed as the product of force and displacement, in Stratonovich form, as
$
 (-\gamma \frac{\ud X_t}{\ud t} + \sqrt{2\gamma k_BT}\frac{\ud B_t}{\ud t}) \circ \ud X_t,
 $ 
 which, using \eqref{eq:overdamped-Langevin}, leads to
\begin{align}
\dbar Q &= \nabla_xU(t,X_t) \circ \ud X_t \label{eq:dQ}\\\nonumber
&=\nabla_xU(t,X_t) \circ(- \frac{1}{\gamma} \nabla_x U (t,X_t)\ud t + \sqrt{\frac{2k_BT(t)}{\gamma}}\ud B_t)\\\nonumber
&=- \frac{1}{\gamma} \|\nabla_x U (t,X_t)\|^2\ud t + \Delta_x U(t,X_t) \frac{k_BT(t)}{\gamma}\ud t\\\nonumber
&\hspace*{2cm}+\nabla_x U (t,X_t) \sqrt{\frac{2k_BT(t)}{\gamma}}\ud B_t,
\end{align}
where the last step includes the correction due to changing into the It\^o form.
Note that we use $\dbar$ to emphasize that $\dbar Q$ is not a perfect differential in that the integral $\int \dbar Q$ depends on the path taken and not just end-point conditions.

The energy exchange between an individual particle and the external potential represents {\em work}. Specifically, the work transferred to the particle by a change in the actuating potential is
\begin{align}
\dbar W &= \frac{\partial U}{\partial t}(t,X_t) \ud t. \label{eq:dW}
\end{align}

Naturally, the {\em first law of thermodynamics}, 
\[
\ud U(t,X_t)=\dbar Q + \dbar W 
\]
holds, since the internal energy is simply the value of the potential.

Accordingly, for a thermodynamic ensemble at a state $\rho(t,x)$, the heat and work differentials are expressed as
\begin{subequations}
\begin{align}\label{eq:dbbQ}
\dbar \mathcal Q &=\left[ \int_{\mR^d} \left(- \frac{1}{\gamma} \|\nabla_x U\|^2 + \Delta_x U \frac{k_BT}{\gamma}\right)\rho\,\ud x\right] \ud t\\\label{eq:dbbW}
\dbar \mathcal W &= \left[ \int_{\mR^d} \frac{\partial U}{\partial t}\rho\,\ud x\right]\ud t,
\end{align}
leading to the first law for the ensemble
\[
\ud \mathcal E(\rho,U) = \dbar \mathcal Q + \dbar \mathcal W,
\]
where the internal energy is
\begin{align}\label{eq:dbbE}
 \mathcal E(\rho,U) &=  \int_{\mR^d} U\rho\,\ud x,
\end{align}
and depends on $\rho,U$, whereas $\mathcal Q,\mathcal W$ depend on the path.
\end{subequations}

\subsection{Summary notation}
As usual, $\mR^d$ denotes the $d$-dimensional Euclidean space, for $d\in \mathbb N$, with $\lr{x}{y}$ and $\|x\|=\sqrt{\lr{x}{x}}$ denoting the respective inner product and norm, for $x,y\in \mR^d$.
For two vector fields $\nabla_x\phi_1,\nabla_x\phi_2$, we denote
$\langle\nabla_x\phi_1,\nabla_x\phi_2\rangle_\rho=\int_{\mR^d}\langle\nabla_x\phi_1,\nabla_x\phi_2\rangle\rho\ud x$, and $\|\nabla_x\phi\|_\rho^2:=\langle\nabla_x\phi,\nabla_x\phi\rangle_\rho$.
The Gaussian distribution with mean $m$ and covariance $\Sigma$ is denoted by $N(m,\Sigma)$.
For convenience we provide Table \ref{table:units} of the various quantities, including the corresponding units in SI format: Newton ($N$), seconds ($s$), meter ($m$), absolute temperature in degrees Kelvin ($\,^o$K).
\begin{table}[t]\label{table:units}
	\centering
\begin{tabular}{|l| l l|}\hline
Definition & Notation & Units\\\hline
time & $t$ & $s$ \\
position of particle & $X_t$ & $m$ \\
Boltzmann constant & $k_B$ & $Nm$ \\
damping coefficient\ & $\gamma$ & $Ns/m$\\
potential & $U(t,x)$ & $Nm$ \\
temperature & $T$ & $\,^o$K\\
Brownian motion & $B_t$ & $s^{\frac{1}{2}}$ \\
density in $\mR^d$& $\rho(t,x)$ & $m^{-d}$ \\
velocity field in $\mR^d$& $v(t,x)$ & $m/s$ \\
Wasserstein length & $\text{length}_{W_2}(\cdot)$ & $m$ \\
entropy & $\cS(\rho)$ & $Nm$ \\
work (particle/ensemble) & $W,\cW$ & $Nm$ \\
heat (particle/ensemble) & $Q,\cQ$ & $Nm$ \\
energy (particle/ensemble) & $U,\cE$ & $Nm$ \\
free energy  & $\cF$ & $Nm$ \\
bound in \eqref{eq:U-constraint}& $M$ & $Nm/s$ \\
power& $P$ & $Nm/s$\\\hline
\end{tabular}
\caption{Symbols and corresponding units}
\end{table}

\section{The second law, dissipation, and Wasserstein geometry}\label{Sec-III-Second-law.tex}

We now discuss the {\em second law of thermodynamics} in the context of an ensemble of particles obeying over-damped Langevin dynamics~\eqref{eq:overdamped-Langevin},  assuming that the temperature of the heat bath remains constant, i.e., $T(t)=T$. The classical formulation amounts to the inequality
	\begin{align}\label{eq:second}
	\cW -\Delta\cF \geq 0
	\end{align}
	where $\cW$ is the work transferred to the ensemble over a time interval $(t_i,t_f)$, namely, 
	\begin{equation*}
	\cW = \int_{t_i}^{t_f} \dbar \cW,
	\end{equation*}
	and $\Delta\cF$ is the change in free energy 
	\begin{equation}
	\cF(\rho,U) = \cE(\rho,U) - T\cS(\rho)
	\end{equation}
between the two end-point states\footnote{The free energy is the amount of energy that can be delivered at  temperature $T$ with fixed potential $U$. A rather revealing re-write of the free energy is as the relative entropy (KL-divergence) between the current state $\rho$ and the Gibbs distribution
$	\rho_{\rm Gibbs}(x)=e^{-\beta U(x)}/Z$, where $\beta=1/k_BT$ and the normalizing factor
$Z=\int_{\mR^d} e^{-\beta U(x)}\ud x$
is the partition function. Specifically, $\cF(\rho,U)=\beta^{-1}\int_{\mR^d} \log(\frac{\rho(x)}{\rho_{\rm Gibbs}(x)}) \rho(x)\ud x - \beta^{-1}\log(Z)$.}.
Here,
	\begin{equation}
	\cS(\rho) = -k_B\int_{\mR^d} \log(\rho)\,\rho\,\ud x
	\end{equation}
	denotes the entropy of the state $\rho$, and $U$ denotes the potential as earlier.   
	
	Inequality \eqref{eq:second} becomes equality for quasi-static (reversible) thermodynamic transitions. In general, for irreversible transitions, the gap in \eqref{eq:second} quantifies dissipation. Interestingly, alternative formulations that shed light into irreversible transitions have recently been discovered. A most remarkable identity was discovered by Jarzynski in the late 90's \cite{jarzynski1997nonequilibrium} to hold for irreversible thermodynamic transitions between work and free energy, in the form,
	\begin{align*}
	\mathbb E \left\{ e^{-\beta W}\right\} - e^{-\beta \Delta \cF}=0,
	\end{align*}
	or, equivalently,
	\begin{align*}
	-\beta^{-1}\log\mathbb E \left\{ e^{-\beta W}\right\} - \Delta \cF=0,
	\end{align*}
	where $\beta^{-1}=k_BT$ and $\Delta \mathcal F$ denotes difference between equilibrium free energy, while the expectation is taken over the probability law on paths.
	
	While the Jarzynski relation establishes equality between the above functional of the work and free energy differences, it does not allow quantifying the actual expected work performed on the ensemble.
	An alternative identity that quantifies explicitly the gap in \eqref{eq:second} holds for irreversible thermodynamic transitions. This identity is
	\begin{subequations}\label{eq:precisedissipations}
		\begin{align}\label{eq:precisedissipation}
		\cW- \Delta \cF = \underbrace{ \gamma \int_{t_i}^{t_f} \big\| 
		\nabla_x\phi(t,\cdot)
		\big\|^2_{ \rho(t,\cdot )} \ud t,}_{\rm dissipation}
		\end{align}
		where
		\begin{align}\label{eq:averagekinetic}
		&
		 \big\| 
		\nabla_x\phi
		\big\|^2_{ \rho}
		= \int_{\mR^d}  \|\nabla_x \phi\|^2 \rho\,\ud x,\\& \frac{\partial \rho}{\partial t}+\nabla_x\cdot (\rho\;\underbrace{\nabla_x \phi}_{v}\,)=0.\label{eq:continuity0}
		\end{align}
	\end{subequations}
	
	The dissipation has the form of an {\em action integral}, since $v=\nabla_x \phi$ is a velocity field that specifies the drift of the ensemble and, thereby, $\|\nabla_x \phi\|^2_\rho$ is the {\em averaged kinetic energy of the ensemble}.
	It is also seen that the dissipation depends solely on the {\em time-parametrized} path
	\begin{align*}
	\rho_{[t_i,t_f]}:=\{\rho(t,\cdot )\mid t\in[{t_i},{t_f}]\}.
	\end{align*}
	Specifically, given any ``tangent''  $\frac{\partial \rho}{\partial t}=\delta$ at any point $\rho(t,\cdot )$ along the path $\rho_{[t_i,t_f]}$, the {\em Poisson equation} \eqref{eq:continuity0} can be solved for $\phi(t,x)$ with $x\in\mR^d$, giving rise to the time-varying, spatially irrotational vector field $v=\nabla_x \phi$ that transitions the thermodynamic system from the starting state $\rho(t,\cdot )$ to $\rho+\delta \ud t$ over the time window $[t,t+\ud t]$.
	
	Interestingly, this recipe of identifying tangent perturbations $\delta$ (i.e., functions $\delta(x)$ such that $\int_{\mR^d}\delta \ud x=0$) with an irrotational field $v=\nabla_x \phi$ via the Poisson equation, instills on the space of probability distributions a Riemannian-like structure via the quadratic form
	\begin{align}\label{eq:innerproduct}
	\int_{\mR^d}  \langle \nabla_x \phi_1,\nabla_x\phi_2\rangle \rho\,\ud x.
	\end{align}
Herein, probability distributions are assumed to have finite second-order moment and be positive. The space of probability distributions (or, measures  \cite{villani2003topics}, with finite second-order moments) with this inner-product structure is known as the Wasserstein manifold $\mathcal P_2(\mR^d)$.
	
	The geodesic distance, i.e., the minimum of
	\begin{align}\label{eq:geodistance}
	\text{length}_{W_2}\left(\rho_{[{t_i},{t_f}]}\right)&=
	\int_{t_i}^{t_f} \big\| \nabla_x\phi(t,\cdot)
	\big\|_{ \rho(t,\cdot )} \ud t,
	\end{align}
	over paths $\rho_{[{t_i},{t_f}]}$ connecting the two end-points
	$\rho_{t_i},\rho_{t_f}$, turns out to be precisely the Wasserstein metric
	\begin{align}\label{eq:W2}
	W_2(\rho_{t_i},\rho_{t_f}):=\sqrt{\inf_{\pi\in\Pi(\rho_{t_i},\rho_{t_f})}\int_{\hspace*{-3pt}\,_{\mR^d\times\mR^d}}\hspace*{-12pt}\|x-y\|^2 \ud \pi(x,y)}
	\end{align}
	of the Monge-Kantorovich theory of Optimal Mass Transport for quadratic transportation cost. Here, $\Pi(\rho_{t_i},\rho_{t_f})$ denotes the probability measures on the product space
	$\mR^d\times\mR^d$ with $\rho_{t_i}(x)$ and $\rho_{t_f}(y)$ as marginals;  see \cite[Chapter 8]{villani2003topics}) and especially \cite{ambrosio2008gradient} for a detailed exposition of the differential structure of $\mathcal P_2(\mR^d)$.
	
	Returning to \eqref{eq:precisedissipations}, 
	provided the averaged kinetic energy \eqref{eq:averagekinetic} during thermodynamic transitions remains constant over time (which can be ensured by a suitable scaling of time of the path $\rho_{[t_i,t_f]}$),
	\begin{align}\label{eq:geodequality}
	\cW- \Delta \cF =\frac{\gamma}{t_f-t_i}\left(\text{\rm length}_{W_2}\left(\rho_{[{t_i},{t_f}]}\right)\right)^2,
	\end{align}
	while in general, the right hand side of \eqref{eq:geodequality} serves as a lower bound. If in addition to constancy of \eqref{eq:averagekinetic} the path is selected as a $W_2$-geodesic, then
	\begin{align}\label{eq:precisedissipation2}
	\cW- \Delta \cF &= \frac{\gamma}{t_f-t_i} W_2(t_i,t_f)^2,
	\end{align}
	which quantifies the maximum amount of work that can be drawn by transitioning between specified end-point thermodynamic states.
	We recap the key points in the following statement.
	\begin{thm}\label{thm1}
		Consider the overdamped model \eqref{eq:overdamped-Langevin} for thermodynamic transitions between states $\rho_{t_i}$, $\rho_{t_f}$, under constant temperature $T$ and a time-varying potential $U$. The following hold:\\[.02in]
		i) In general,
		\begin{align}\label{eq:precisedissipation1}
		\cW- \Delta \cF \geq \frac{\gamma}{t_f-t_i}\left(\text{\rm length}_{W_2}\left(\rho_{[{t_i},{t_f}]}\right)\right)^2.
		\end{align}
		ii) Relation \eqref{eq:precisedissipation1} holds as the equality in \eqref{eq:geodequality}
		for a suitable time-reparametrization of the path $\rho_{[{t_i},{t_f}]}$ of the thermodynamic ensemble, effected by a suitable choice of potential. \\[.02in]
		iii) There is a unique path $\rho_{[{t_i},{t_f}]}$ ($W_2$-geodesic) for the thermodynamic transition that attains minimal dissipation and, in this case, \eqref{eq:precisedissipation2} holds.
	\end{thm}
	
	\begin{proof}
		We first derive \eqref{eq:precisedissipations}: consider  
		\begin{align*}
		\frac{\ud \cF}{\ud t}(\rho,U)=& \frac{\ud}{\ud t} \cE(\rho,U)  - T \frac{\ud}{\ud t}S(\rho)\\=& \frac{\ud}{\ud t} \int_{\mR^d}  U \rho\ud x +\kbt \frac{\ud}{\ud t} \int_{\mR^d} \rho \log \rho\;\ud x\\
		=& 
		\int_{\mR^d} \left(\frac{\partial{U}}{\partial{t}} \rho +U\frac{\partial{\rho}}{\partial{t}}\;+\kbt\frac{\partial{\rho}}{\partial{t}} \log \rho \right)\; \ud x\\
		=&~\int_{\mR^d}\frac{\partial{U}}{\partial{t}}\rho\;\ud x+\int_{\mR^d}\left(U+\kbt \log \rho\right)\frac{\partial{\rho}}{\partial{t}}\ud x.
		\end{align*}
		Using the Fokker-Planck equation~\eqref{eq: Fokker-Planck}, the second term
		\begin{align*}
		&\int_{\mR^d}  \left(U+\kbt \log \rho\right)\frac{1}{\gamma}\nabla_x\cdot \left[ (\nabla_x U +k_BT\nabla_x \log \rho)\rho\right] \ud x\\
		&=-\frac{1}{\gamma}\int_{\mR^d}\|\nabla_x U + \kbt\nabla_x \log\rho\|^2 \rho \;\ud x \\
		&=-\gamma\int_{\mR^d}\|v\|^2 \rho \;\ud x,
		\end{align*}
		where the first equality follows using integration by parts (under standard assumptions on the decay rate of $\rho$ at infinity), while the second equality is a re-write using 
		\begin{equation}\label{eq:v-def}
		v := -\frac{1}{\gamma}(\nabla_x U + \kbt\nabla_x \log\rho).
		\end{equation}
		Thus,
		\begin{align*}
		\frac{\ud \cF}{\ud t}(\rho,U)\!=\!\int_{\mR^d} \frac{\partial{U}}{\partial{t}}\rho\;\ud x-\gamma\int_{\mR^d}\|v\|^2 \rho\;\ud x.
		\end{align*}	
		Integrating over $[t_i,t_f]$ yields
		\begin{align}\label{eq:intermediary}
		\Delta \cF&= \cW-\gamma\int_{t_i}^{t_f} \int_{\mR^d}\|v\|^2 \rho\;\ud x\;\ud t,
		\end{align}	
		where $v$ is the gradient of $\phi = - \frac{1}{\gamma}(U + \kbt \log \rho)$ and satisfies the continuity equation~\eqref{eq:continuity0} as claimed.   This establishes  \eqref{eq:precisedissipations}.
		
		Statements i) and ii) follow from the fact that 
		the $W_2$-length of the path $\rho_{[t_i,t_f]}$ (i.e., as a curve in $\mathcal P_2$), is given by \eqref{eq:geodistance}. Specifically,
		provided  $
		\int_{\mR^d}\|v\|^2 \rho\;\ud x=\alpha^2$ remains constant along the path (i.e., for $t\in[t_i,t_f]$),
		\[
		\alpha= \frac{1}{t_f-t_i}\text{length}_{W_2}\left(\rho_{[{t_i},{t_f}]}\right)
		\]
		and \eqref{eq:geodequality} follows from \eqref{eq:intermediary}. 
		If on the other hand the kinetic energy varies with time, then the path $\rho(t,\cdot )$,  time-reparametrized by
		\[
		\tilde t(t):= \frac{\text{\rm length}_{W_2}\left(\rho_{[{t_i},{t}]}\right)}{\text{\rm length}_{W_2}\left(\rho_{[{t_i},{t_f}]}\right)}(t_f-t_i)+t_i
		\]
		will be traversed via a velocity field 
		\[\tilde v(\tilde t(t))=\frac{v(t)}{
			{\|v(t)\|_\rho}
			}\frac{\text{\rm length}_{W_2}\left(\rho_{[{t_i},{t_f}]}\right)}{t_f-t_i}.\]
		Knowing $\tilde v$, a new potential $\tilde U$ can be computed so that
		$\tilde v(\cdot,\tilde t)=\nabla_x\tilde U(\cdot, \tilde t)+k_BT\nabla_x\rho(\cdot, \tilde t)$.
		
		Finally, statement iii) follows by taking $\rho_{[t_i,t_f]}$ to be a geodesic.
	\end{proof}
	
	\begin{remark} Early work by Jordan etal.\ \cite{jordan1998variational}, pointing out that the gradient flow of the free energy in $W_2$ is the Fokker-Planck equation, set the stage for understanding the role of the Wasserstein geometry in quantifying dissipation. This was recognized in \cite{aurell2011optimal,aurell2012refined,seifert2012stochastic}
		and more recently developed in \cite{chen2019stochastic,dechant2019thermodynamic}.
	\end{remark}

\section{Cyclic operation of engines}\label{Sec-IV-cycle.tex}
\label{sec:cycle}
We consider two types of thermodynamic transitions, isothermal and adiabatic. The first corresponds to a situation where the system remains in contact with a heat bath of constant temperature $T$ while a time-varying potential steers its thermodynamic state $\rho(t,.)$ from an {\em initial} $\rho(t_i,\cdot)$ to a {\em final} $\rho(t_f,\cdot)$. The adiabatic transition amounts to abrupt changes in both, the temperature of the heat bath as well as the shape of the potential, that are fast enough not to have any measurable effect on the state $\rho(t,.)$ and, as a consequence, to the entropy of the ensemble. We evaluate next the energy and work budgets in the correspoinding actuation protocols.

\subsection{Isothermal transition}
We consider transition between states $\rho_i$ and $\rho_f$ for the ensemble modeled by \eqref{eq:overdamped-Langevin}, over a time interval $[t_i,t_f]$, under the time-varying potential $U(t, X_t)$ and in contact with a heat bath of temperature $T$. Using the relationship~\eqref{eq:precisedissipation} between work, free energy, and the dissipation, and the first law, we have the following identity relating thermodynamic quantities in isothermal transitions
\begin{subequations}\label{eq:isothermalbalance}
	\begin{align}\label{eq:balance}
\cW&=\Delta\cE -T \Delta \cS + \cW_{\rm irr}\\
\cQ&=T \Delta \cS - \cW_{\rm irr}\
	\end{align}
	with the {\em irreversible} $\cW_{\rm irr}$ that represents dissipation attaining its minimal value
	\begin{equation}\label{eq:minWirr}
	\frac{\gamma}{t_f-t_i}W_2(\rho_{t_i},\rho_{t_f})^2
	\end{equation}
\end{subequations}
by the choice of actuation $\nabla_x U(t,\cdot)$ in \eqref{eq:v-def} with $v$ the optimal velocity field minimizing dissipation in \eqref{eq:precisedissipations} (item iii) in Theorem \ref{thm1}). 

It is important to note that the minimizing $v$ can be obtained by solving a convex reformulation of \eqref{eq:precisedissipations} in terms of the density $\rho(t,\cdot)$ and the momentum field $\mathbf p(t,\cdot)=v(t,\cdot)\rho(t,\cdot)$, in the form
\begin{subequations}
\begin{align}\label{eq:BB}
&\min_{\mathbf p(t,\cdot),\rho(t,\cdot)} \int_{t_i}^{t_f} \int_{\mR^d} \frac{\|\mathbf p\|^2}{\rho}\ud x\ud t\\
&\mbox{ subject to } \frac{\partial \rho}{\partial t}+\nabla_x \cdot \mathbf p = 0\\
&\mbox{ and } \rho(t_i,\cdot),\rho(t_f,\cdot)\mbox{ specified.}
\end{align}
\end{subequations}
Then, $v=\mathbf p/\rho$, see
\cite[Section 4]{benamou2000computational} and \cite[p.\ 241]{villani2003topics}.

\subsection{Adiabatic transition}

We now consider transition between $\rho_i$ and $\rho_f$ for the ensemble modeled by  \eqref{eq:overdamped-Langevin}, over a time interval $[t_i,t_f]$, under abrupt changes in the potential $U(t, X_t)$ and the temperature $T$ of the heat bath. 

The transition takes place over an infinitesimally short time interval about time $t$ (with $t^-/t^+$ indicating the left/right limits, respectively).
Thus, the temperature $T$ of the heat bath jumps between values $T({t^-})$ and $T({t^+})$ while, at the same time, the controlling potential switches from $U({t^-},\cdot)$ to $U({t^+},\cdot)$.

The energy budget of the transition no longer contains irreversible losses, as the right hand side of~\eqref{eq:precisedissipation} vanishes. Moreoverm, the entropy of the ensemble remains constant because $\rho({t^+},\cdot)=\rho({t^-},\cdot)$. Thus, the work input into the system equal to change in internal energy,
\begin{subequations}\label{eq:adiabaticbalance}
	\begin{align}
 \cW &=\int_{\mR^d} (U(t^+,x)-U(t^-,x))\rho(t,x)\ud x
	 = \Delta \cE,
	\label{eq:work-adiabatic}
	\end{align} 
	and therefore no heat transfer takes place, and therefore,
	\begin{align}
\cQ=0. \label{eq:heat-adiabatic}
	\end{align}
\end{subequations}

\subsection{Finite-time Carnot cycle}

We are now in position to consider a complete {\em Carnot-like thermodynamic cycle} where the ensemble is steered between two states $\rho_a$ and $\rho_b$ during isothermal expansion (from $\rho_a$ to $\rho_b$) and contraction (from $\rho_b$ to $\rho_a$) phases, separated by adiabatic transitions. Periodic operation about such a scheduling is sought as a means to extract work from a heat bath. A schematic in Figure~\ref{fig:cycle} depicts the phases of the cyclic operation. These four phases are described in detail next. 

\begin{center}
	\begin{figure}[t]
		\includegraphics[width=.7\linewidth]{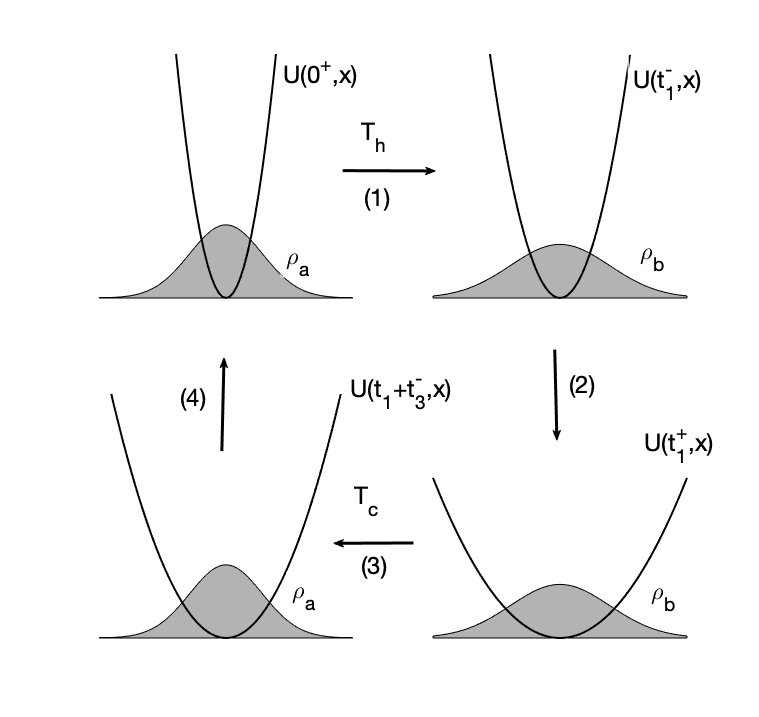}
		\centering
		\caption{
			Carnot-like cycle of a stochastic model for a heat engine (with $d=1$): the operation cycles clockwise through two isothermal transitions $(1)$ and $(3)$,  and two adiabatic transitions $(2)$ and $(4)$. During the isothermal transitions having duration $t_1$ and $t_3$, the ensemble is in contact with a ``hot'' reservoir of temperature $T_h$, and a ``cold'' one of temperature $T_c$, respectively. The adiabatic transitions are considered to be instantaneous, i.e., $t_2=t_4= 0$. The marginal densities are $\rho_a$ and $\rho_b$.}
		\label{fig:cycle}
	\end{figure}
\end{center}

\medskip
\noindent
{\bf 1)  Isothermal process in temperature $T_h$ (``hot''):} The first step is an isothermal expansion over the time interval $(0,t_1)$ in contact with a heat bath of temperature $T=T_h$. Change in the potential steers the ensemble from a starting state $\rho_a$ to a terminal state $\rho_b$. As in \eqref{eq:isothermalbalance},
\begin{subequations}\label{eq:phase1}
	\begin{align}
	\cW^{(1)}&=\Delta\cE^{(1)} -T_h \Delta \cS^{(1)}+\cW_{\rm irr}^{(1)}\\
	\cQ^{(1)} &=T_h \Delta \cS^{(1)}-\cW_{\rm irr}^{(1)}\label{eq:dQ1}
	\end{align}
	where the superscript enumerates the phase in the cycle, and the minimal work loss
	$\cW_{\rm irr}^{(1)}$ depends only on the end-point states as it equals
	\begin{equation}
	\cW_{\rm irr}^{(1)} = \frac{\gamma}{t_1}W_2(\rho_a,\rho_b)^2.
	\end{equation}
\end{subequations}

\medskip
\noindent
{\bf 2)  Adiabatic process:} The second phase of the cycle is an adiabatic transition at time $t=t_1$, over an infinitesimal interval (of duration``$t_2= 0$''), bringing the ensemble in contact with a heat bath of temperature $T_c$ (``cold''). As in \eqref{eq:adiabaticbalance},
\begin{subequations}\label{eq:phase2}
	\begin{align}
	\cW^{(2)}&=\Delta\cE^{(2)}\\
	\cQ^{(2)}&=0
	\end{align}
\end{subequations}
while the state remains at $\rho_b$.

\medskip
\noindent
{\bf 3)  Isothermal process in temperature $T_c$ (``cold''):} The third step is an Isothermal contraction over the time interval $(t_1, t_1+t_3)$ while in contact with a heat bath of temperature $T_c$. Actuation in the form of the time-varying potential causes the state of the ensemble to  return to $\rho_a$ back from starting at $\rho_b$. Once again, as in 
\eqref{eq:isothermalbalance},
\begin{subequations}\label{eq:phase3}
	\begin{align}
	\cW^{(3)}&=\Delta\cE^{(3)} -T_c \Delta \cS^{(3)}+\cW_{\rm irr}^{(3)}\\
	\cQ^{(3)} &=T_c \Delta \cS^{(3)}-\cW_{\rm irr}^{(3)}\\
	\cW_{\rm irr}^{(3)} &= \frac{\gamma}{t_3}W_2(\rho_a,\rho_b)^2.
	\end{align}
\end{subequations}

\medskip
\noindent
{\bf 4)  Adiabatic process:} Finally, an adiabatic transition over an interval of infinitesimal duration (``$t_4= 0$'') returns the ensemble to be in contact with a heat reservoir of temperature $T_h$ for a total period of the cycle $t_{\rm period}=t_1+t_3$. The state of the ensemble remains at $\rho_a$, to begin the cycle again. As before, in \eqref{eq:adiabaticbalance},
\begin{subequations}\label{eq:phase4}
	\begin{align}
	\cW^{(4)}&=\Delta\cE^{(4)}\\
	\cQ^{(4)}&=0
	\end{align}
\end{subequations}

\subsection{Thermodynamic efficiency \& power delivered}
For a cyclic process the total  change in internal  energy 
\[
\sum_{i=1}^{4} \Delta \cE^{(i)}=0.
\]
On the other hand, the entropy doesn't change during the adiabatic transitions
\[
\Delta \cS^{(i)}=0, \mbox{ for }i=2,4,
\]
while, since it depends only on the end-point states $\rho_a,\rho_b$,
\[
\Delta S^{(1)}=-\Delta S^{(3)}= \cS(\rho_b)-\cS(\rho_a)=:\Delta \cS.
\]
As a result, the total work output (negative of the work input) is
\begin{equation}
\begin{aligned}\label{eq:total-work-initial}
- \cW&=-\left(\sum_{i=1}^{4}\Delta \cE^{(i)}-\sum_{i=1}^{4}T_i \Delta \cS^{(i)}+\sum_{i=1}^{4} \cW_{irr}^{(i)}\right)\\
&=(T_h-T_c)\Delta \cS-\cW_{\rm irr}^{(1)}-\cW_{\rm irr}^{(3)}.
\end{aligned}
\end{equation}
Thus, assuming optimality of the choice of the potential to minimize $\cW_{\rm irr}$ in each transition,
we conclude that the total work output possible is
\begin{equation}
\begin{aligned}\label{eq:total-work-initial2}
 -\cW&=(T_h-T_c)\Delta \cS-\gamma (\frac{1}{t_1}+\frac{1}{t_3})W_2(\rho_a,\rho_b)^2.
\end{aligned}
\end{equation}
Since $T_h>T_c$, naturally, a necessary condition for positive work output is that \[\Delta \cS := \cS(\rho_b)-\cS(\rho_a)>0\] 
which dictates that phase 1 is an isothermal expansion and phase 3, an isothermal contraction.\footnote{The opposite would be true if we sought to operate the cycle for refrigeration purposes.}

The thermodynamic efficiency of an engine is the ratio of work extracted over the heat dissipated, 
\begin{equation}\label{eq:efficiency}
\eta = \frac{-\cW}{\phantom{.}\mathcal Q_h}
\end{equation}
where the heat input during isothermal expansion, from \eqref{eq:dQ1}, is
\[
\cQ_h= \Delta \cQ^{(1)}= T_h\Delta \cS-\cW_{\rm irr}.
\]
Once again assuming optimality ($\cW_{\rm irr}=\frac{\gamma}{t_1}W_2(\rho_a,\rho_b)^2$), the bound on the efficiency is seen to be
\begin{equation}
\begin{aligned}\label{eq:boundefficiency}
\eta = \frac{(T_h-T_c)\Delta \cS-\gamma (\frac{1}{t_1}+\frac{1}{t_3})W_2(\rho_a,\rho_b)^2}{T_h\Delta \cS-\gamma \frac{1}{t_1}W_2(\rho_a,\rho_b)^2}
\end{aligned}
\end{equation}
When the period of the cyclic process tends to infinity (and hence, $t_1,t_3\to\infty$), tends to the Carnot limit for quasistatic (infinitely slow) transitions
\[
\eta_{\rm C}=1-\frac{T_c}{T_h}.
\]

Periodic operation, over a finite period $t_1+t_3$ (since $t_2=t_4=0$), delivers
\begin{align}\nonumber
P &= -\cW/(t_1+t_3)\\
&=\frac{(T_h-T_c)\Delta \cS-\gamma (\frac{1}{t_1}+\frac{1}{t_3})W_2(\rho_a,\rho_b)^2}{t_1+t_3}
\label{eq:power}
\end{align}
units of power. Note that the power output is zero when Carnot efficiency is achieved, because the total duration $t_1+t_3 \to \infty$.  In the sequel, we focus on assessing bounds on available power.

\section{Fundamental limits to power}\label{Sec-V-power.tex}

Our main interest is in assessing the maximal amount of power that can be drawn by a thermodynamic engine operating between heat baths with temperatures $T_h$ and $T_c<T_h$, i.e.,``hot'' and ``cold'', respectively. 
In the present work we draw conclusions based on the basic model in \eqref{eq:overdamped-Langevin} via analysis of the thermodynamic cycle that was presented in Section \ref{sec:cycle}.

Consider the expression in \eqref{eq:power} for the power that can be drawn via a cyclic operation as discussed. Preparation of the ensemble, and actuation during the cycle, allow a number of choices. Specifically, the power depends on the period $t_1+t_3$, the times of the two isothermal phases $t_1,t_3$ individually, as well as the end-point states (distributions) $\rho_a,\rho_b$. The latter choice impacts both, the Wasserstein distance $W_2(\rho_a,\rho_b)$ as well as the change in entropy $\Delta\cS$. We will explore systematically the various options.

\subsection{Optimizing the time scheduling}\label{sec:51}
Optimizing the maximal power delivered during cyclic operation
\begin{align}\nonumber
P&=\frac{1}{t_1+t_3}(T_h-T_c)\Delta \cS-\frac{\gamma}{t_1t_3}W_2(\rho_a,\rho_b)^2,
\end{align}
with respect to choices for $t_1,t_3$, with $W_2(\rho_a,\rho_b)$, $T_h,T_c$ and $\Delta \cS$ kept fixed, gives that
\begin{equation}\label{eq:split}
t_1=t_3=\frac{4\gamma W_2(\rho_a,\rho_b)^2}{(T_h-T_c)\Delta \cS},
\end{equation}
and therefore that the period for the cycle is
\begin{align}\nonumber
t_{\rm cycle}&:=t_1+t_3\\\label{eq:t1t3free}
&=\frac{8\gamma W_2(\rho_a,\rho_b)^2}{(T_h-T_c)\Delta \cS}.
\end{align}

If instead we specify the period of the cycle $t_{\rm cycle}$, and optimize with respect to the breakdown between $t_1$ and $t_3$, we once again obtain that the durations of the two phases are equal
\begin{equation}\label{eq:halfs}
t_1=t_3=\frac{t_{\rm cycle}}{2}.
\end{equation}
\begin{remark}[Efficiency at maximum power] The thermodynamic efficiency~\eqref{eq:efficiency} of the engine, when it is operating at optimal transition times~\eqref{eq:split} that maximize the power,  is equal to
	\begin{equation}\label{eq:eta-max-power}
	\eta_{SS} = \frac{2(T_h-T_c)}{3T_h+T_c} = \frac{\eta_C}{2-\frac{\eta_C}{2}}
	\end{equation}
	The result~\eqref{eq:eta-max-power} first appeared in~\cite{seifert2008stochastic}, for special setting when the two marginal distributions $\rho_a$ and $\rho_b$ are Gaussians, and the potential $U(t,x)$ is a quadratic function of $x$. Our derivation verifies the result~\eqref{eq:eta-max-power} in a general setting.
\end{remark}

Using the expression~\eqref{eq:t1t3free}, the total power delivered
\begin{equation}
P = \frac{(T_h-T_c)^2}{16\gamma } \left(\frac{\Delta \cS}{W_2(\rho_a,\rho_b)}\right)^2.
\end{equation}
But as we will see in Section \ref{sec:caveat}, optimizing the power for $\rho_a,\rho_b$ leads to the non-physical conclusion of a vanishingly small $t_{\rm cycle}$ .

\subsection{The caveat of optimal $t_{\rm cycle}$: Gaussian states $\rho_a,\rho_b$}\label{sec:caveat}

The case where the two marginal distributions/states are Gaussian allows for closed-form expressions for $\Delta\cS$ and their Wasserstein distance. Indeed, if  $\rho_a,\rho_b$ are Gaussian distributions with zero mean and variances $\Sigma_a,\Sigma_b$, respectively, then
\begin{subequations}\label{eq:gaussianformulae}
	\begin{align}\label{eq:W2-Gaussian}
	W_2^2(\rho_a,\rho_b)&=\tr\left(\Sigma_a+\Sigma_b-2(\Sigma_{a}^{1/2}\Sigma_b\Sigma_a^{1/2})^{1/2}\right),\\\nonumber
	\Delta\cS&=\cS(\rho_b)-\cS(\rho_a)\\\label{eq:entropyGaussian}
	&=\frac12k_B \log \det (\Sigma_b\Sigma_a^{-1}).
	\end{align}
\end{subequations}
Evidently, these allow deriving explicit expressions for the available power in terms of the respective variances.

Specializing to the case of scalar processes with $\sigma_i$ ($i\in\{a,b\}$) the corresponding standard deviation, i.e., $\Sigma_i=\sigma_i^2$, and period $t_{\rm cycle}$ for the thermodynamic cycle as in \eqref{eq:t1t3free}, we obtain that the maximal power available, as a function of $\sigma_a$ and $\sigma_b$, is given by
\begin{equation}
\label{eq.maxpwrtt}
P(\sigma_a,\sigma_b)=\frac{k_B^2(T_h-T_c)^2}{16\gamma }\left(\frac{\log \frac{\sigma_b}{\sigma_a}}{\sigma_b-\sigma_a}\right)^2.
\end{equation}
The corresponding heat uptake from the hot reservior and the work extracted during one cycle are computed as
\begin{align*}
\cQ^{(1)}=\cQ_{h}=\frac{1}{4}k_B(3T_h+T_c)\log \frac{\sigma_b}{\sigma_a}
\end{align*}
and
\begin{align*}
-\cW=\frac{1}{2}k_B(T_h-T_c)\log \frac{\sigma_b}{\sigma_a},
\end{align*}
respectively.

The maximum of the power $P(\sigma_a,\sigma_b)$ over either $\sigma_a$, or $\sigma_b$, takes place when 
\[
\sigma_a=\sigma_b.
\]
But at this limiting condition, although
\begin{subequations}
	\begin{equation}\label{eq:max-power-Gaussian}
	\begin{aligned}
	\max_{\sigma_b}P(\sigma_a,\sigma_b)&=\frac{k_B^2(T_h-T_c)^2}{16\gamma \sigma_a^2}\end{aligned}
	\end{equation}
	and the rate with which heat is drawn is
	\begin{equation*}
	\begin{aligned}\lim_{\sigma_b\to\sigma_a}\frac{\cQ_h}{t_{\rm cycle}}&= \frac{k_B^2(3T_h+T_c)(T_h-T_c)}{32\gamma \sigma_a^2},
	\end{aligned}
	\end{equation*}
	the limiting values of
	$-\Delta \cW$, $\cQ_{h}$ over a cycle vanish, as does the period $t_{\rm cycle}$ of the cycle.
\end{subequations}
Thus we are led to a non-physical situation of a vanishingly small period for the thermodynamic cycle.

A similar issue in the context of power in quantum engines is brought up in \cite{esposito2010quantum}. In the setting herein, in addition, it is seen that taking
\[
\sigma_a\to 0
\]
and operating with a vanishingly small period for the cycle, leads to infinite power. Once again, bringing up a non-practical situation that is questionable on physical grounds. In the sequel we focus on $t_{\rm cycle}$ being finite.

\subsection{Optimizing the thermodynamic state $\rho_b$}\label{sectionVC}

Henceforth we fix the period $t_{\rm cycle}$ as well as the duration of the isothermal phases according to \eqref{eq:halfs}. The power delivered, as a function of the $\rho_i$'s ($i\in\{a,b\}$), is
\begin{equation}\label{Pbar}
\frac{(T_h-T_c)}{t_{\rm cycle}} (\cS(\rho_b)-\cS(\rho_a))- \frac{4\gamma }{t_{\rm cycle}^2}W_2(\rho_a,\rho_b)^2.
\end{equation}

We now consider the problem to maximize power over choice of $\rho_b$, with $\rho_a$ specified. This problem reduces to finding a suitable minimizer of
\begin{align}\label{eq:findrhob}
\min_{\rho_b} \{W_2(\rho_a,\rho_b)^2 - h\cS(\rho_b),\}
\end{align}
for
\[
h =\frac{t_{\rm cycle}(T_h-T_c)}{4\gamma }.
\]

Throughout we assume that states have finite second-order moments. As noted earlier, the space of probability distributions (measures, in general) with finite second-order moments $\mathcal P_2(\mathbb R^d)$ is metrized by the Wasserstein metric $W_2(\cdot,\cdot)$ and, as can easily be verified, the expression
\begin{equation}\label{eq:jko}
W_2(\rho_a,\rho_b)^2-h\cS(\rho_b)
\end{equation}
is strictly convex, which leads to the following statement.

\begin{proposition}\label{prop:prop1} Assuming that $T_h,T_c$ as well as $t_{\rm cycle}$ and an initial state $\rho_a\in\mathcal P_2(\mathbb R^d)$ are specified, there exists a unique minimizer $\rho_b$ of \eqref{eq:findrhob}.
\end{proposition}

\begin{proof} 
	Equation \eqref{eq:findrhob} is similar to one step in the so-called JKO-scheme (also, proximal projection) that displays the heat equation as the gradient flow of the Shannon entropy \cite{jordan1998variational}. 
	While $W_2(\rho_a,\rho_b)^2-h\cS(\rho_b)$ is strictly convex, it is not automatically bounded from below. Thus, a rather extensive and technical argument is needed to show existence and uniqueness of a minimizer. This is detailed in \cite[Proposition 4.1]{jordan1998variational}.
\end{proof}

We conclude this section with two statements. The first establishes implicit conditions for optimality of $\rho_b$, in maximizing the expression in \eqref{Pbar} (equivalently, minimizing \eqref{eq:jko}). For ease of referencing we view the expression in \eqref{Pbar} as a function of $\rho_b$, namely,
\begin{align}\label{eq:f}
f(\rho_b):=\frac{(T_h-T_c)}{t_{\rm cycle}} (\cS(\rho_b)-\cS(\rho_a))- \frac{4\gamma}{t_{\rm cycle}^2}W_2(\rho_a,\rho_b)^2.
\end{align}
The following lemma provides stationarity conditions for $f(\rho_b)$ that, albeit, are implicit in that they involve the optimal transport map from $\rho_a$ and $\rho_b$ that minimizes quadratic transportation cost \cite[Ch. 5]{villani2003topics}. 

The theory of optimal transport provides that, since $\rho_a,\rho_b$ are densities (as opposed to measures), the support of $\Pi$ in \eqref{eq:W2}
coincides with the graph of a map 
\[
\Psi:\mR^d\to\mR^d\,:\,x\mapsto y,
\]
which in fact is the gradient  of a convex function $\psi$ on $\mathbb R^d$ \cite[Ch. 5]{villani2003topics}, i.e., $\Psi=\nabla \psi$. This is the (unique) optimal transport map for the so-called {\em Monge problem} to minimize  $\int_{\mR^d} \|x-\Psi(x)\|^2 \rho_a(x)\ud x$ over all maps that transfer mass from $\rho_a$ to $\rho_b$.  The transferance of the ``mass'' distribution $\rho_a$ into $\rho_b$ is indicated by
\[
\nabla \psi \sharp \rho_a = \rho_b,
\]
which is a compact notation for the change of variables formula
\[
\det(\nabla^2\psi(x))\rho_b(\nabla\psi(x))=\rho_a(x).
\]
We first highlight stationarity conditions that characterize the minimizer of $f(\cdot)$ in \eqref{eq:f}.

\begin{lemma} \label{prop:opt-pb} Consider two probability densities $\rho_a,\rho_b^*$ in $\mathcal P_2(\mathbb R^d)$, where $\rho_b^{*}$ is the unique maximizer of $f(\rho_b)$, and let $\nabla \phi$, for a convex function $\psi$ on $\mathbb R^d$, be such that
	$\nabla \psi\sharp \rho_a =\rho_b^*$.
	The following (stationarity) condition holds
	\begin{equation}
	\label{optimalrhob}
	k_B(T_h-T_c)\nabla \log \rho_b^{*}(y)-\frac{8\gamma }{t_{\rm cycle}}\left((\nabla \psi)^{-1}-{\rm Id}\right)(y)=0,
	\end{equation}
	where $\rm Id$ denotes the identity map.
\end{lemma}

\begin{proof}The proof is given in Appendix \ref{apdx:opt-pb}. \end{proof}

The lemma, which is of independent interest, is used in the proof of the following proposition which concludes the section.
The proposition states that, for scalar distributions for simplicity, if $\rho_a$ is Gaussian, then so is $\rho_b$. As a consequence the optimal actuation protocol is based on a time-varying potential $U(t,x)$ that is quadratic in $x$. 
\begin{proposition}	
	\label{prop:optgau}
	If $\rho_a$ is a one-dimensional Gaussian distribution with zero mean and variance $\sigma_a^2$, then $\rho_b^{*}$ is also Gaussian with zero mean and variance $\sigma_b^2$, where
	\begin{align}
	\sigma_b&=\frac{1+\sqrt{1+c}}{2}\sigma_a,	\label{eq:opt-sigmab}
	\end{align}
	and $c=\frac{k_B(T_h-T_c)t_{\rm cycle}}{2 \gamma \sigma_a^2}$.                    
\end{proposition}
\begin{proof}The proof is given in Appendix \ref{apdx:optgau}. \end{proof}

\begin{remark}
	In earlier works, it is commonly assumed that the marginal distributions $\rho_a$, $\rho_b$ are Gaussian and the potential function $U(t,x)$ is quadratic in $x$. 
	Proposition~\ref{prop:optgau} justifies this assumption to some extent: if $\rho_a$ is specified to be Gaussian, the optimal $\rho_b$ and the optimal potential function that achieve the maximum power,  are Gaussian and quadratic, respectively. However, as we will see in Section~\ref{sec:opt-pa}, if instead $\rho_b$ is specified as Gaussian distribution, the optimal $\rho_a$ is not Gaussian. Gaussian distributions turn out instead to be local {\em minimizers} of the power under certain conditions (see discussion following Remark~\ref{rem:pa-minimizer}).
\end{remark}
\subsection{Optimizing the thermodynamic state $\rho_a$}\label{sec:opt-pa}
We now consider the dependence of the maximal power on $\rho_a$, i.e., on the thermodynamic state at which the ensemble begins its expansive phase. As we will see, the situation is not symmetric to the conclusions drawn in Section \ref{sectionVC} with regard to $\rho_b$ and, without further assumptions, an optimal $\rho_a$ does not exist. Interestingly, on closer inspection, the source of this conundrum is the unreasonably high demands on the magnitude of $\nabla_xU$ for the controlling potential $U(t,x)$.
The insights gained lead to the framework for maximal power in the follow up section.

For simplicity, and without any loss of generality for the purposes of this section, we assume that $\rho_b$ is specified to be a zero-mean Gaussian distribution with standard deviation $\sigma_b$. In view of \eqref{Pbar}, a choice of $\rho_a$ that is close to a Dirac delta distribution allows arbitrarily large negative values for the entropy, i.e., $\cS(\rho_a)\simeq - \infty$, and hence infinite power.

Thus, it is natural to impose a lower bound on the entropy of $\rho_a$, or simply fix
\[
-\infty < s_a=\cS(\rho_a) <\cS(\rho_b).
\]
But in this case, and once more in view of \eqref{Pbar}, maximal power would be drawn by minimizing $W_2(\rho_a,\rho_b)$ over probability densities $\rho_a$ with entropy $s_a$. We claim that
\begin{equation}\label{eq:infW2ra}
\inf_{\rho_a}\{W_2(\rho_a,\rho_b)\mid \cS(\rho_a)=s_a>-\infty\} = 0.
\end{equation}
To see this note that
\[
\inf_{\rho_a}W_2(\rho_a,\rho_b) = 0
\]
by taking $\rho_a$ to approximate an increasingly fine train of suitably scaled Dirac deltas, i.e., 
\[
\rho_a(x)\approx \sum_{i\in \mathbb Z} \rho_i\delta_{x_i}(x)
\]
where $\rho_i=\int_{x_i}^{x_{i+1}}\rho_b(x)\ud x$ and $x_i$ ($i\in \mathbb Z$) equispaced. The latter is a singular distribution which, however, can be approximated arbitrarily closely in $W_2$ by a probability density with any given entropy.
Such a density can be produced by approximating Dirac deltas by a piecewise constant function with finite support. 

The optimization problem~\eqref{eq:infW2ra} is inherently related to the continuity of the entropy functional with respect to the Wasserstein distance. For a rigorous treatment of the problem, see~\cite{polyanskiy2016wasserstein}, where it is shown that unless certain regularity assumptions are in place for $\rho_a$ and $\rho_b$, the infimum in~\eqref{eq:infW2ra} is zero.

\begin{remark}[Gaussian is not optimal for $\rho_a$]
\label{rem:pa-minimizer}
The preceeding arguments show that  a Gaussian distribution is not the optimal choice for $\rho_a$ with respect to maximizing power, even when $\rho_b$ is Gaussian, unless additional constraints are introduced. 
\end{remark}

Since the Gaussian distribution maximizes entropy when mean and variance are specified,  it is natural to explore constraints on the mean and variance of $\rho_a$ for the purposes of maximizing power. Without loss of generality, the mean can be assumed to be zero and the variance specified to be $\sigma_a^2<\sigma_b^2$. First-order and second order optimality analysis for the power output~\eqref{Pbar}, at $\rho_a=N(0,\sigma_a^2)$ is carried out. It turns out that, although $N(0,\sigma_a^2)$ satisfies the first-order optimality condition, it does not satisfy the second-order optimality condition.  In fact, $N(0,\sigma_a^2)$  is a local minimizer when $\sigma_a<\sigma_b< k_B (T_h-T_c) t_{\rm cycle}/(8\gamma\sigma_a)$. The analysis is given in Appendix~\ref{apdx:pa-minimizer}, and aims to highlight that the conjecture that a Gaussian $\rho_a$ is optimal fails.

\subsection{Maximum power with arbitrary potential}\label{sec:55}
In this section, we show that the power output of a thermodynamic engine, under any choice of potential $U(t,x)$ cannot exceed a bound that involves the Fisher information of  the marginal state $\rho_a$. The Fisher information 
is defined as
\begin{align*}
I(\rho):=\int_{\rho>0} \frac{\|\nabla \rho\|^2}{\rho} \ud x.
\end{align*}
\begin{proposition} Under the standing assumptions on the Carnot-like cycle,
the power output~\eqref{Pbar}, is bounded by 
\begin{align}\label{eq:power-bound-HWI}
P
&\leq \frac{k_B^2 (T_h-T_c)^2}{16\gamma} I(\rho_a).
\end{align}
\end{proposition}

\begin{proof} This is based on the following HWI inequality (see \cite{villani2008optimal,gentil2019entropic}  for details), 
\begin{equation*}
\cS(\rho_b)-\cS(\rho_a) \leq k_BW_2(\rho_a, \rho_b)\sqrt{I(\rho_a )},
\end{equation*}
Using the formula for power~\eqref{Pbar}, we have
	\begin{align*}
	 P&=\frac{(T_h-T_c)\Delta \cS}{\tf}-\frac{4\gamma }{\tf^2}W_2(\rho_a,\rho_b)^2\\
	&\leq \frac{(T_h-T_c)\Delta \cS}{\tf}-\frac{4\gamma}{\tf^2}\frac{\Delta \cS^2}{k_B^2I(\rho_a)}\\
	&=\!-\frac{4\gamma}{\tf^2}\!\frac{1}{k_B^2I(\rho_a )}\!\left(\!\Delta S\!-\!\frac{\tf k_B^2(T_h-T_c)}{8\gamma}I(\rho_a)\!\right)^2\\
	&\;\;\;\;\;\;+\frac{k_B^2(T_h-T_c)^2}{16\gamma}I(\rho_{a})\\
	&\leq \frac{k_B^2(T_h-T_c)^2}{16\gamma}I(\rho_{a}),
	\end{align*}
	concluding the bound~\eqref{eq:power-bound-HWI}. \end{proof}
	
	We point out that the bound~\eqref{eq:power-bound-HWI} is achieved when $\tf$ is given by its optimal value~\eqref{eq:t1t3free} and in the limit as $\rho_b \to \rho_a$. In particular, suppose $\rho_a$ and $\rho_b$ are Gaussian distributions $N(0,\sigma_a^2)$ and $N(0,\sigma_b^2)$, respectively, and that the cycle period $\tf$ is equal to the optimal value~\eqref{eq:t1t3free}, then as $\sigma_b \to \sigma_a$ the power output is given by~\eqref{eq:max-power-Gaussian}, which is equal to the bound~\eqref{eq:power-bound-HWI}, because $I(\rho_a)= \frac{1}{\sigma_a^2}$. 
	
\subsection{Maximum power under constrained potential}\label{sec:56}
While a lower bound on $\cS(\rho_a)$ readily implies an upper bound on the available power, achieving such a bound in general requires a cyclic operation involving an irregular and complicated potential function $U(t,x)$ to bring back the ensemble to $\rho_a$ at end of each cycle. It is unreasonable to expect technological solutions to such demands, and therefore, a constraint on the complexity of the potential function seems meaningful. 
To this end, we propose the constraint 
\begin{equation}\label{eq:U-constraint}
\frac{1}{\gamma}\int_{\mR^d}\|\nabla_x U(t,x)\|^2 \rho(t,x)\,\ud x \leq M
\end{equation}
for all $t\in(0,\tf)$. Thus, we analyze the maximum power~\eqref{Pbar} that can be extracted from a thermodynamic engine, under the constraint~\eqref{eq:U-constraint}.
\begin{thm} \label{thm:power-bound}
Consider a thermodynamic ensemble, undergoing a Carnot cycle as described in Section~\ref{sec:cycle}, governed with the over-damped Langevin equation~\eqref{eq:overdamped-Langevin}. Then, the maximum power $P$ that can be extracted from the cycle, over all marginal probability distributions  $\rho_a$ and $\rho_b$, the cycle period $\tf$, and all potential functions $U(t,x)$ that respect the bound~\eqref{eq:U-constraint},  satisfies
 \begin{equation}\label{eq:power-bound-U}
\frac{M}{8} (\frac{T_h}{T_c}-1) \frac{\frac{T_h}{T_c}-1}{\frac{T_h}{T_c}+1}\leq  P_{\rm max}\leq \frac{M}{8} (\frac{T_h}{T_c}-1)
 \end{equation}
\end{thm}

\begin{proof}
The proof for the upper-bound follows from bounding the entropy difference $S(\rho_b)-S(\rho_a)$ under the constraint~\eqref{eq:U-constraint}. During the isothermal transition in contact with the cold bath with temperature $T_c$, 
\begin{align*}
&\cS(\rho_b)-\cS(\rho_a) =   \cS(\rho(\frac{\tf}{2},\cdot))- \cS(\rho(\tf,\cdot)) \\
&=-\int_{\frac{\tf}{2}}^{\tf} \frac{\ud}{\ud t} \cS(\rho(t,\cdot))\,\ud t\\
&=k_B\int_{\frac{\tf}{2}}^{\tf}\int_{\mR^d} \log \rho(t,x) \frac{\partial \rho}{\partial t } (t,x)\,\ud x\, \ud t\\
&=\frac{-k_B}{\gamma}\int_{\frac{\tf}{2}}^{\tf}\int_{\mR^d} \lr{\nabla_x\log \rho}{\nabla_x U +k_BT_c\nabla_x \log \rho}\rho\,\ud x \ud t\\
&=\frac{-k_B}{\gamma}\int_{\frac{\tf}{2}}^{\tf} (\langle \nabla_x\log \rho, \nabla_x U\rangle_\Ltwo +k_BT_c \|\nabla_x \log \rho\|_\Ltwo^2 )\ud t,
\end{align*}
where with a slight abuse of notation we use $\langle \nabla_xf,\nabla_xg\rangle_\rho$ to also denote $\int_{\mR^d}\langle \nabla_xf,\nabla_xg\rangle \rho\ud x$.
By the Cauchy-Schwartz inequality and constraint~\eqref{eq:U-constraint}, 
\begin{align*}
- \langle \nabla_x \log \rho, \nabla_x U\rangle_\Ltwo&\leq \|\nabla_x U\|_\Ltwo  \|\nabla_x \log \rho\|_\Ltwo\\&\leq \sqrt{\gamma M}  \|\nabla_x \log \rho\|_\Ltwo.
\end{align*}
Hence, 
\begin{align*}
&\cS(\rho_b)-\cS(\rho_a) \\ \leq &\frac{k_B}{\gamma}\int_{\frac{\tf}{2}}^{\tf} \left(\sqrt{\gamma M}  \|\nabla_x \log \rho\|_\Ltwo -k_BT_c \|\nabla_x \log \rho\|_\Ltwo^2 \right)\ud t
\\ \leq &\frac{k_B}{\gamma}\int_{\frac{\tf}{2}}^{\tf}\frac{\gamma M}{4k_BT_C} \ud t=\frac{M}{8T_C} \tf.
\end{align*} 
This concludes  the bound $\Delta \cS \leq \frac{M}{T_C}\frac{\tf}{8}$ on the entropy difference, which yields to  upper-bound on the power output:
\begin{equation}
P\leq \frac{(T_h-T_c)}{\tf} \Delta \cS - \frac{1}{\tf}\Wirr \leq  \frac{M(T_h-T_c)}{8T_C}
\end{equation}
where $\Wirr\geq 0$ is used. 

Next, we prove the lower-bound by describing a setting so that the power is equal to the lower bound. Assume the marginal distributions $\rho_a$ and $\rho_b$ are Gaussian $N(0,\sigma_a^2)$ and $N(0,\sigma_b^2)$ respectively, and the potential function $U(t,x)=\frac{1}{2}a_tx^2$ is a quadratic function. In this setting, the exact power output is equal to
\begin{align*}
P=&\frac{1}{\tf}k_B(T_h-T_c)\log(\frac{\sigma_b}{\sigma_a})\\&- \frac{1}{\gamma\tf} \int_0^{\tf}( a_t - \frac{k_B T}{\sigma_t^2})^2 \sigma_t^2 \ud t 
\end{align*}
with update law for the variance given by the Lyapunov equation:
\begin{align*}
&\frac{\ud \sigma_t^2}{\ud t} = -2(\frac{a_t}{\gamma} -  \frac{k_B T}{\gamma \sigma_t^2} ) \sigma_t^2
\end{align*}
with the constraint~\eqref{eq:U-constraint} given by 
\begin{align*}
\frac{1}{\gamma}a_t^2\sigma_t^2 \leq M
\end{align*}
Then, in the limit as $\tf \to 0$,  and $\sigma_a=\sigma_b=\sigma$, the power output is equal to 
\begin{equation}
\begin{aligned}
&P=k_B(T_h-T_c)\frac{\lambda}{2} - \gamma  \lambda^2 \sigma^2
\end{aligned}
\label{eq:opt-power-t}
\end{equation}
with the constraint 
\begin{equation}\label{eq:lambda-constraint}
|\gamma \lambda + \frac{k_BT_C}{\sigma^2}| \leq \frac{\sqrt{\gamma M}}{\sigma},
\end{equation}
where we introduced a new variable $\lambda = \frac{a}{\gamma} - \frac{k_BT_c}{\gamma \sigma^2}$. It is shown in Appendix~\ref{apdx:lower-bound}, that the maximum of the expression~\eqref{eq:opt-power-t} over all values of $\lambda $ and $\sigma$ that satisfy the constraint~\eqref{eq:lambda-constraint}, is equal to 
\begin{equation*}
\frac{M}{8} (\frac{T_h}{T_c}-1)\frac{\frac{T_h}{T_c}-1}{\frac{T_h}{T_c}+1}
\end{equation*}
concluding the lower-bound.
\end{proof}

This final result is universal as it does not depend on the choice of $\rho_a$ and $\rho_b$, unlike \eqref{eq:power-bound-HWI}. Moreover, the bounds in this final result are especially appealing in that it the depend on the ratio $T_h/T_c$ of the absolute temperatures of the two heat baths.

\begin{remark}
It is noted that the upper bound in \eqref{eq:power-bound-U} on achievable power under the constraint \eqref{eq:U-constraint} does not depend on $\tf$, whereas our construction for achieving the lower bound ensures that the bound is approached as $\tf\to 0$.
\end{remark}

\begin{remark}
In the proof of Theorem \ref{thm:power-bound}, an operating point has been constructed to ensure that power equal the lower bound in \eqref{eq:power-bound-U} can be achieved. The parameters are given in equation \eqref{eq:oneline} in the Appendix. For this operating point, which corresponds to maximal power constrained by \eqref{eq:U-constraint}, the efficiency
turns out to be
\[
\eta=\frac{T_h-T_c}{T_h+T_c}.
\]
It is interesting to note that
\[
\eta_{SS}\leq \eta_{CA}\leq \eta \leq \eta_C,
\]
where $\eta_{SS}$ is the efficiency obtained by Schmiedl and Seifert given in \eqref{eq:eta-max-power}, $\eta_{CA}=1-\sqrt{T_c/T_h}$ is the Curzon-Ahlborn efficiency, and $\eta_C=1-T_c/T_h$ is the Carnot Efficiency. Furthermore, $\eta_{CA},\eta$ and $\eta_C$ tend to $1$ as $T_c\to 0$, while $\eta_{SS}\to 2/3$.
\end{remark}

\section{Concluding remarks}

The present work focused on quantifying the maximal power that can be drawn by a Carnot-like heat engine operating by alternating contact with two heat reservoirs and modeled by stochastic overdamped Langevin dynamics driven by the time dependent potential. The framework that the work is based on is that of Stochastic Thermodynamics~\cite{seifert2008stochastic,sekimoto2010stochastic,seifert2012stochastic,parrondo2015thermodynamics,dechant2016underdamped}, which allows quantifying energy and heat exchange by individual particles in a thermodynamic ensemble, to be subsequently averaged, so as to quantify performance of the thermodynamic process as a whole. 
A physically reasonable bound is derived, which is shown to be reached within a specified factor, both depending on the ratio $T_h/T_c$ of the absolute temperatures of the two heat baths, hot and cold, respectively. The present work is quite distinct from earlier results, within a similar framework, which is however restricted to Gaussian states. Conditions that suggest non-physical conclusions are highlighted, and a suitable constraint on the controlling potential is brought forth that underlies our analysis.

In the past few decades, there have been several attempts to quantify efficiency mainly, but also power, of thermodynamic processes operating in Carnot-like manner. It is fair to say that there has been neither a consensus on the type of assumptions that have been used by previous authors, and thereby, nor full consistency of the results.
This is to be expected, since finite-period operation and finite-time thermodynamic transitions require substance/engine dependent assumptions to capture the complexity of heat transfer in non-equilibrium states. Thus, estimated bounds may never reach the ``universality'' of the celebrated Carnot efficiency. They are expected to provide physical insights and guidelines for engineering design. Thus, it will be imperative that these estimates be subject to experimental testing. The notable feature of our conclusions as compared to earlier works is that the expressions we derive are given in the form of ratio of absolute temperatures--a physically suggestive feature.

The present work follows a long line of contributions within the control field to draw links between thermodynamics and control, see e.g., \cite{brockett1979stochastic,pavon1989stochastic,mitter2005information,sandberg2014maximum,rajpurohit2017stochastic,wallace2014thermodynamics}.
More recently,
important insights have linked the Wasserstein distance of optimal mass transport, which itself is a solution to a stochastic control problem, to the dissipation mechanism in stochastic thermodynamics
\cite{aurell2011optimal,aurell2012refined,seifert2012stochastic,chen2019stochastic,dechant2019thermodynamic}. Indeed, the Wasserstein metric takes the form of an action integral and arises naturally in the energy balance of thermodynamic transitions. This fact has been explored and developed for the overdamped Langevin dynamics studied herein. Whether similar conclusions can be drawn for underdamped Langevin dynamics remains an open research direction at present. Furthermore, much work remains to reconcile and compare alternative viewpoints and models for thermodynamic processes including those based on the Boltzmann equation.

Besides the potentially intrinsic value of the analysis and bounds that have been derived, it is hoped that the control-theoretic aspect of the problem to optimize Carnot-like cycling of thermodynamic process has been sufficiently highlighted, and that this work will serve to raise attention on this important and foundational topic to the control community.
\appendix

		\subsection{Proof of Lemma~\ref{prop:opt-pb}}\label{apdx:opt-pb}
					Consider an arbitrary smooth vector field with bounded support $\xi \in C_{0}^{\infty}(\mR^d, \mR^d)$. Let $\Psi_s:\mR^d \to\mR^d$ be the map generated by $\xi$ defined according to  
			\begin{align*}
			\frac{\partial}{\partial s}\Psi_s(x)=\xi (\Psi_s(x)),\quad \Psi_{0}={\rm Id},
			\end{align*}
			for $x \in \mR^d$ and $s \in \mR^d$. Then, define 
			\begin{equation*}
			\rho_s := \Psi_s \sharp \rho_b^*.
			\end{equation*}
			We claim that 
			\begin{equation}\label{eq:claim}
			\begin{aligned}
			&\lim_{s \to 0} \frac{1}{s} (f(\rho_s) -f(\rho_b^*)) \geq  \int  \langle D_f(x),\xi(x)\rangle \rho_b^*(x)\ud x,
			\end{aligned}
			\end{equation}
			where, for $\Delta T:=T_h-T_c$,
			\begin{equation*}
			D_f(x)= -\frac{k_B\Delta T}{t_{\rm cycle}} \nabla \log(\rho^*_b(x)) + \frac{8\gamma}{t_{\rm cycle}^2} (\nabla \psi^{-1}(x)-x).
			\end{equation*}
			Assuming the claim is true (to be shown shortly), then,  because $\rho^*_b$ is the maximizer, $f(\rho_s) \leq  f(\rho^*_b) $.   Therefore 
			\begin{align*}
		 \int  \langle D_f(x),\xi(x)\rangle \rho_b^*(x)\ud x& \leq \lim_{s \to 0} \frac{ f(\rho_s)\!\! -\!\! f(\rho_b^*)}{s}
			\leq 0.
			\end{align*}
			Hence, by symmetry $\xi \to -\xi$,  
			\begin{equation}
		 \int  \langle D_f(x),\xi(x)\rangle \rho_b^*(x)\ud x = 0.
			\end{equation}
			This is true for all vector fields $\xi \in C^\infty_0(\mR^d,\mR^d)$. As a result, $D_f(x)=0$, concluding~\eqref{optimalrhob} and the lemma. 
			
			It remains to prove~\eqref{eq:claim}. By definition, the difference $f(\rho_s) -  f(\rho^*_b)$ is 
			\begin{align*}
			f&(\rho_s) -  f(\rho^*_b) = 
			\frac{\Delta T}{t_{\rm cycle}}(\cS(\rho_s) - \cS(\rho_b^*)) \\&- \frac{4\gamma}{t_{\rm cycle}^2}(W_2(\rho_a,\rho_s)^2 - W_2(\rho_a,\rho_b^*)^2).
			\end{align*}
			The entropy  term
			\begin{align*}
			\cS(\rho_s) =& -k_B\int  \log(\rho_s(x)) \rho_s(x)\ud x  \\
			=&-k_B \int \log(\rho_s(\Psi_s(x))) \rho_b^*(x) \ud x \\
			=&-k_B \int \log(\frac{\rho^*_b((x))}{\text{det}(\nabla \Psi_s(x))}) \rho_b^*(x) \ud x \\
			=& S(\rho_b^*)  +k_B\int \log(\text{det}(\nabla \Psi_s(x))) \rho_b^*(x) \ud x.
			\end{align*}
			Therefore
			\begin{align*}
			&\lim_{s \to 0} \frac{1}{s} (S(\rho_s)- S(\rho_b^*)) \\=&  \lim_{s \to 0} \frac{k_B}{s} \int  \log(\text{det}(\nabla \Psi_s(x))) \rho_b^*(x) \ud x \\
			=&k_B \int \nabla \cdot \xi(x) \rho^*_b(x)\ud x\\
			= & -k_B\int \langle \xi(x), \nabla \log(\rho^*_b(x))\rangle \rho^*_b(x)\ud x.
			\end{align*}
			The Wasserstein term 
			\begin{align*}
			&W_2(\rho_a,\rho_s)^2 - W_2(\rho_a,\rho_b^*)^2\\ 
			\leq  &\int\|\nabla \psi^{-1}(x) - \Psi_s(x) \|^2 \rho_b^{*}(x)\ud x \\&- \int\|\nabla \psi^{-1}(x) - x\|^2 \rho_b^{*}(x)\ud x\\= &\int\langle x-\Psi_s(x),2\nabla \psi^{-1}(x) -x-\Psi_s(x)\rangle\rho_b^{*}(x)\ud x.
			\end{align*}
			Therefore
			\begin{align*}
			&\lim_{s\to 0}\frac{1}{s}\left[W_2(\rho_a,\rho_s)^2 - W_2(\rho_a,\rho_b^*)^2\right]\\
			\leq&-2 \int \langle \xi(x),\nabla \psi^{-1}(x) -x\rangle \rho_b^{*}(x)\ud x.
			\end{align*}
			Using  the two expressions, the one for derivative of the entropy and the other for the Wasserstein distance, the claim follows.  
			
		\subsection{Proof of Proposition \ref{prop:optgau}}\label{apdx:optgau}
					According to Proposition~\ref{prop:prop1}, the maximizer is unique. Therefore, it is sufficient to show that the Gaussian distribution $N(0,\sigma_b^2)$, where $\sigma_b^2$ is given by \eqref{eq:opt-sigmab}, satisfies the optimality condition~\eqref{optimalrhob}. 
			When $\rho_a, \rho_b^{*}$ are Gaussian, $\nabla \psi^{-1}(y)=\frac{\sigma_a}{\sigma_b}y$. Hence, the optimality condition reads 
			\begin{align*}
			& \frac{k_B t_{\rm cycle}\Delta T}{8\gamma}\nabla \log \rho_b^{*}(y)-y+\nabla \psi^{-1}(y)\\
			=& \frac{k_B t_{\rm cycle}\Delta T}{8\gamma}\frac{y}{\sigma_b^2}-(1-\frac{\sigma_a}{\sigma_b})y\\
			=&(\frac{k_B t_{\rm cycle}\Delta T}{8\gamma\sigma_b^2}-1+\frac{\sigma_a}{\sigma_b})y=0,\quad \forall \; y \in \mathbb{R},
			\end{align*}
			which is satisfied when
			$\sigma_b$ is according to~\eqref{eq:opt-sigmab}. 
		
\subsection{Proof of statements following Remark~{\ref{rem:pa-minimizer}}}	\label{apdx:pa-minimizer}	
Let $\calA_{0,\sigma^2}$ denote the set of absolutely continuous distributions with mean $0$ and variance $\sigma^2$ :
\begin{align*}
\calA_{0,\sigma^2}&:=\bigg \{\rho \in \Pac;\\&~ \int x \rho(x) \ud x=0, \int x^2 \rho(x) \ud x=\sigma^2 \bigg\},
\end{align*}
and consider the functional
\begin{equation}
g(\rho_a) = -\frac{T_h-T_c}{t_{\rm cycle}}S(\rho_a)-\frac{4\gamma}{t_{\rm cycle}^2}W_2(\rho_a,\rho_b)^2
\end{equation}
that represents the portion of power given in~\eqref{Pbar} that depends on $\rho_a$. 
\begin{romannum}
	\item We first show that \[\max_{\rho_a \in \calA_{0,\sigma_a^2}}g(\rho_a)\] is unbounded, and hence that the maximizer does not exist. 
	Consider a sequence of density functions $\{\mu_n\}_{n \in \mathbb{N}}$ according to
	\begin{align*}
	\mu_n\ =\frac{1}{2}\mu^{(1)}_n+\frac{1}{2}\mu^{(2)}_n,
	\end{align*}
	where
	\begin{equation}
	\begin{aligned}
	\mu^{(1)}_n&=N(\sqrt{\sigma_a^2-\frac{1}{n^2}}, \frac{1}{n^2})\\
	\mu^{(2)}_n&=N(-\sqrt{\sigma_a^2-\frac{1}{n^2}}, \frac{1}{n^2})
	\end{aligned}
	\end{equation}
	It is easy to verify that $\mu_n \in \mathcal{A}_{0, \sigma_a^2}$. 
	The goal is to show that $\lim_{n\to\infty}g(\mu_n)=\infty$. 
	The entropy  $\cS(\mu_n)$ is bounded from above as follows, 
	\begin{equation*}		
		\begin{aligned}
	\cS(\mu_n) & \leq \frac{1}{2}(k_B\log(2)\! +\! \cS(\mu^{(1)}_n))\! + \!\frac{1}{2}(k_B\log(2)\! +\! \cS(\mu^{(2)}_n))\\&=k_B\log 2+\frac{k_B}{2}\log(2\pi e \frac{1}{n^2})\\
	&=k_B\log 2\sqrt{2\pi e}- k_B\log n.
	\end{aligned} 
	\end{equation*}
	The first inequality follows from a respective bound on the entropy of Gaussian mixtures~\cite[Theorem 3]{huber2008entropy}.
	
	The Wasserstein distance $W_2(\mu_n, \rho_b)^2$ is bounded by
	\begin{equation}
	\begin{aligned}
	W_2(\mu_n, \rho_b)^2&\leq \frac{1}{2}W_2(\mu^{(1)}_n, \rho_b)^2+\frac{1}{2}W_2(\mu_n^{(2)}, \rho_b)^2\\
	&=W_2(\mu_n^{(1)}, \rho_b)^2\\
	&=(\sqrt{\sigma_a^2-\frac{1}{n^2}}-0)^2+(\frac{1}{n}-\sigma_b)^2\\
	&=\sigma_a^2+\sigma_b^2-2\frac{\sigma_b}{n},
	\end{aligned}
	\end{equation}
	where the convexity of the functional $W_2^2(\cdot,\rho_b)$ is used~\cite[Eq.(2.12)]{carlen2003constrained}. 
	Combining the two bounds for the entropy and Wasserstein distance yields
	\begin{equation*}
	\begin{aligned}
	g&(\mu_n)=-\frac{\Delta T}{\tf}\cS(\mu_n)-\frac{4}{\tf^2}W_2(\mu_n,\rho_b)^2\\
	&\geq \frac{k_B\Delta T}{\tf}\left(\!-\!\log 2\sqrt{2\pi e}+\log n\!\right)\\
	& \;\;\;\;-\frac{4}{\tf^2}\left(\sigma_a^2+\sigma_b^2-2\frac{\sigma_b}{n}\right).
	\end{aligned}
	\end{equation*}
	Taking the limit $n \to \infty$ proves\ $\lim_{n\to\infty}g(\mu_n)$ is unbounded, and hence that there is no maximizer. Next, in (ii) and (iii) we show that the Gaussian distribution is instead a local minimizer, under certain conditions, and hence that it is the opposite of distributions that we seek.\\[.02in]

\item We now perform first-order optimality analysis for the problem 
\[\max_{\rho_a \in \calA_{0,\sigma_a^2}}g(\rho_a)\]  at $\rho_a=N(0,\sigma_a^2)$.  
Consider a smooth vector field with bounded support $\xi \in C_{0}^{\infty}(\mR, \mR)$ such that 
\begin{equation}\label{eq:mean-xi}
\int_{\mR^d} \xi(x) \rho_a(x) dx=0.
\end{equation} 
Then, as before, define the flow $\Psi_s:\mR \to \mR$ generated by $\xi$ according to  
\begin{align*}
\frac{\partial}{\partial s}\Psi_s(x)=\xi (\Psi_s(x)),\quad \Psi_{0}={\rm Id},
\end{align*}
for $x \in \mR$ and $s \in \mR$.  Now define
\begin{equation*}
\tilde{\rho}_s := \Psi_s\sharp \rho_a.
\end{equation*}
Because of~\eqref{eq:mean-xi}, the mean of the distribution $\tilde{\rho}_s$ remains constant at $0$. However, the variance changes from $\sigma_a^2$, and $\tilde{\rho}_s$ is not inside $\calA_{0,\sigma_a^2}$.

In order to keep the variance constant at $\sigma_a^2$, we project $\tilde{\rho}_s$ into $\calA_{0,\sigma_a^2}$ by setting
\begin{equation*}
\rho_s = G_s \sharp\tilde{\rho}_s,
\end{equation*}
where 
$
G_s(x): =r(s)x
$
and
\begin{align*}
r(s)=\frac{\sigma_a}{\sqrt{\int_{\mR^d}|y|^2\tilde{\rho}_s(y)\ud y}}=\frac{\sigma_a}{\sqrt{\int _{\mR^d} |\Psi_s(x)|^2 \rho_a(x) \ud x}}.
\end{align*}
By  definition of $G_s$, the variance of $\rho_s$ is equal to $\sigma_a^2$, hence $ \rho_s \in \calA_{0,\sigma_a^2}$. 

Following the same procedure as in the proof of Lemma \ref{prop:opt-pb}, the first-order optimality condition is
\begin{equation}\label{eq:opt-cond}
\int \langle D_g(x), v(x)\rangle \rho_a(x)\ud x = 0,
\end{equation}
where
\begin{equation}\label{eq:v-def-apdx}
\begin{aligned}
v(x) &= \frac{\ud }{\ud s} G_s(\Psi_s(x))\vert_{s=0}\\&=\xi(x) - \frac{x}{\sigma_a^2}\int \xi(z)z\rho_a(z)\ud z,
\end{aligned}
\end{equation}
and
\begin{equation*}
D_g(x)= \frac{k_B\Delta T}{\tf} \nabla \log(\rho^*_a(x)) + \frac{8\gamma}{\tf^2} (\nabla \psi(x)-x)
\end{equation*}
where $\nabla \psi$ is the optimal transport map from $\rho_a$ to $\rho_b$. For the setting where $\rho_a$ and $\rho_b$ are $N(0,\sigma_a^2)$ and $N(0,\sigma_b^2)$, respectively, $\nabla \psi(x) = \frac{\sigma_b}{\sigma_a}x$ and 
\begin{equation}\label{eq:Dg}
D_g(x)=- \frac{k_B\Delta T}{\tf} x+ \frac{8\gamma}{\tf^2} (\frac{\sigma_b}{\sigma_a}-1)x.
\end{equation}
Letting $\alpha:=- \frac{k_B\Delta T}{\tf} + \frac{8\gamma}{\tf^2} (\frac{\sigma_b}{\sigma_a}-1)$ and inserting \eqref{eq:Dg}
into~\eqref{eq:opt-cond} yields
\begin{equation*}
\alpha  \int x v(x) \rho_a(x)\ud x = 0,
\end{equation*}
which is satisfied because of~\eqref{eq:v-def-apdx}.  Hence, $\rho_a$ being $N(0,\sigma_a^2)$ satisfies the first-order optimality condition.\\[.02in]

\item We follow up by carrying out second-order analysis. The objective is to show that the limit
\begin{align*}
&\lim_{s \to 0} \frac{1}{s^2}\left(g(\rho_s) + g(\rho_{-s}) - 2g(\rho_a)\right)=\\
&-\frac{\Delta T}{\tf}\lim_{s \to 0} \frac{\cS(\rho_s) + \cS(\rho_{-s}) - 2
	\cS(\rho_a)}{s^2}\\
&-\frac{4\gamma }{\tf^2}\lim_{s \to 0} \frac{W_2(\rho_s,\rho_b)^2 + W_2(\rho_{-s},\rho_b)^2 - 2W_2(\rho_a,\rho_b)^2}{s^2}
\end{align*}
can be strictly positive.
Assume $\xi(x) = \nabla \eta(x)$ for some $\eta$, and define 
\[
\zeta(x)=\eta(x)-\frac{x^2}{2\sigma_a^2}\int_{\mR} z\nabla \eta (z)\rho_a(z) \ud z,
\]
For the second order derivative of the entropy, we use the existing results from~\cite[Eq. (2.30), Eq.(3.37)]{carlen2003constrained},
where it is shown that
\begin{align}		\label{eq:ddS}
\frac{\ud^2}{\ud s^2}S(\rho_s)\bigg \vert_{s=0}&=-k_B\int \left(\|\nabla^2 \zeta\|_F^2+ \frac{1}{\sigma_a^2}\|\nabla \zeta\|^2\right)\rho_a \ud x
\end{align} 

Next, we consider the  second order derivative of the Wasserstein distance. Since $\nabla \psi \sharp \rho_a = \rho_b$ and  $(G_s \circ \Psi_s)  \sharp \rho_a=\rho_s$,  we have
\begin{align*}
W_2(\rho_b, \rho_s)^2 \leq \int_{\mR} |\nabla \psi(x)-G_s (\Psi_s(x))|^2  \rho_a(x) \ud x.
\end{align*}
As a result
\begin{align*}
&\lim_{s \to 0} \frac{W_2(\rho_s, \rho_b)^2+W_2(\rho_{-s}, \rho_b)^2-2W_2(\rho_a,\rho_b)^2}{s^2}\\
\leq &\lim_{s \to 0} \frac{1}{s^2}\bigg[ 
\int | \nabla \psi (x)-G_s(\Psi_s(x))|^2\rho_a(x)\ud x +\\
&\quad\quad+\int | \nabla \psi (x)-G_{-s}(\Psi_{-s}(x))|^2\rho_a(x)\ud x\\
&\quad\quad-2\int | \nabla \psi (x)-x|^2\rho_a(x)\ud x\bigg]\\
=&\lim_{s \to 0} \frac{1}{s^2}\bigg[ \int \left(|G_s(\Psi_s(x))|^2+|G_{-s}(\Psi_{-s}(x))|^2\right) \rho_a(x)\ud x 
\\&\quad\quad\quad-2\int \left(|x|^2+\langle \Omega_s(x), \nabla \psi (x)\rangle \right) \rho_a(x)\ud x \bigg],
\end{align*} 
where $\Omega_s(x) = G_s (\Psi_s(x))\!+\!G_{-s}(\Psi_{-s}(x))\!-\!2x$.  The first three terms cancel out, because the variance is constant. Therefore, the limit simplifies to
\begin{equation}
\label{eq:W2-dt2-temp-1}	
\begin{aligned}
&\lim_{s \to 0} \frac{W_2(\rho_s, \rho_b)^2+W_2(\rho_{-s}, \rho_b)^2-2W_2(\rho_a,\rho_b)^2}{s^2}\\
\leq&-2\!\int_{\mR^d}\lim_{s \to 0} \frac{\Omega_s(x)}{s^2}\cdot \nabla \psi (x)\rho_a(x)
\ud x\\
=&-2\!\!\int _{\mR^d}\frac{\partial^2 G_s (\Psi_s(x))}{\partial s^2}\bigg \vert_{s=0} \nabla \psi (x) \rho_a(x) \ud x\\
=&-2\frac{\sigma_b}{\sigma_a} \int_{\mR^d}\frac{\partial ^2 G_s (\Psi_s(x))}{\partial{s^2}}(x)\bigg \vert_{s=0} x \rho_a(x) \ud x,
\end{aligned}
\end{equation}
where $\nabla \psi (x)=\frac{\sigma_b}{\sigma_a}x$ is used in the last step. 
Next we compute $\frac{\partial ^2 G_s (\Psi_s(x))}{\partial{s^2}}(x) \vert_{s=0} $. Differentiating once gives
\begin{align*}
\frac{\partial{G_s (\Psi_s(x))}}{\partial{s}}&= \frac{\partial} {\partial s} (r(s) \Psi_s(x))\\&=r(s) \nabla \eta \left(\Psi_s(x)\right)+\dot{r}(s)\Psi_s(x).
\end{align*}
Differentiating twice and evaluating at $s=0$ gives
\begin{align*}
&\frac{\partial^2{G_s (\Psi_s(x))}}{\partial{s^2}}\bigg \vert_{s=0}\\
= &2\dot{r}(0) \nabla \eta (x)+r(0)\nabla^2 \eta(x)\nabla \eta(x)+\ddot{r}(0)x.
\end{align*}
Inserting this expression into~\eqref{eq:W2-dt2-temp-1} gives
\begin{equation}
\label{eq:W2-dt2-temp-2}	
\begin{aligned}
&\lim_{s \to 0} \frac{W_2(\rho_s, \rho_b)^2+W_2(\rho_{-s}, \rho_b)^2-2W_2(\rho_a,\rho_b)^2}{s^2}\\
\leq&-2\frac{\sigma_b}{\sigma_a} \bigg[2\dot{r}(0) \int x \nabla \eta(x)\rho_a(x)\ud x  \\
&+  \int x\nabla^2\eta(x)\nabla \eta(x) \rho_a(x)\ud x + \sigma_a^2\ddot{r}(0)\bigg].
\end{aligned}
\end{equation}
Inserting the derivatives of $r(s)$,
\begin{align*}
\dot{r}(0) &=- \frac{1}{\sigma_a^2}\int \nabla \eta(x)x\rho_a(x)\ud x,\\
\ddot{r}(0)&= \frac{3}{\sigma_a^4}\left(\int \nabla \eta(x)x\rho_a(x)\ud x\right)^2\\
&-\frac{1}{\sigma_a^2}\int (|\nabla \eta(x)|^2 + x\nabla^2\eta(x)\nabla \eta(x))\rho_a(x)\ud x,
\end{align*}
gives
\begin{align}\nonumber
&\lim_{s \to 0} \frac{W_2(\rho_s, \rho_b)^2+W_2(\rho_{-s}, \rho_b)^2-2W_2(\rho_a,\rho_b)^2}{s^2}\\\nonumber
\leq&-\frac{2\sigma_b}{\sigma_a} \bigg[\frac{1}{\sigma_a^2}\left(\int x\nabla \eta(x)\rho_a(x)\ud x\right)^2-\int |\nabla \eta |^2\rho_a\ud x\bigg]\\
=& \;\;2\frac{\sigma_b}{\sigma_a}\int |\nabla \zeta(x)|^2\rho_a(x)\ud x.\label{eq:W2-dt2-temp-3}	
\end{align}
Using \eqref{eq:ddS} and \eqref{eq:W2-dt2-temp-3}, we conclude that
\begin{align*}
&\lim_{s \to 0} \frac{1}{s^2}\left(g(\rho_s) + g(\rho_{-s}) - 2g(\rho_a)\right)\\\geq\;\;
& 
(\frac{k_B\Delta T}{\tf\sigma_a^2} - \frac{8\gamma \sigma_b}{\tf^2\sigma_a}) \int  \|\nabla \zeta\|^2 \rho_a \ud x.
\end{align*}
Hence, when 
$\sigma_b \in (\sigma_a, \frac{k_B \Delta T\tf}{8\gamma\sigma_a}]$, the second-order variation is positive and $\rho_a=N(0,\sigma_a^2)$ is a local minimizer. 
\end{romannum}

\subsection{Proof of the lower-bound in Theorem~\ref{thm:power-bound}}\label{apdx:lower-bound}
	The constraint \eqref{eq:lambda-constraint} is expressed as: 
\begin{equation*}
0 \leq  \lambda \leq  \frac{\sqrt{\gamma M}}{\gamma \sigma} - \frac{k_BT_c}{\gamma \sigma^2}, \mbox{ for }\sigma\geq\frac{k_BT_c}{\sqrt{\gamma M}}.
\end{equation*}
The inequality $\lambda\geq 0$ ensures that the power is non-negative, whereas $\sigma\geq\frac{k_BT_c}{\sqrt{\gamma M}} $ ensures that the upper bound is positive.   
We utilize dimensionless variables
\begin{align*}
x&:=\frac{\lambda}{\lambda_0},\quad
y:=\frac{\sigma_0}{\sigma} 
\end{align*}
for 
$
\sigma_0 := k_BT_c/\sqrt{\gamma M}$, $\lambda_0:=M/k_BT_c 
$, and re-write \eqref{eq:opt-power-t} and the constraints,
\begin{align*}
&P=Mf(x,y)\\
&0\leq x\leq g(y),\quad 0<y\leq 1
\end{align*}
where 
$
f(x,y)=\frac{\Delta T}{2T_C}x - \frac{x^2}{y^2}$, $
g(y) = y-y^2$. As long as $y\leq y_0$, where $y_0= \frac{1}{1+\frac{\Delta T}{4T_C}}$, the unconstrained maximizer
\begin{equation*}
x^*(y) = \argmax _x~f(x,y) = \frac{\Delta T}{4T_C}y^2
\end{equation*}
satisfies the constraint $x^*(y)\leq g(y)$. When $y_0<y\leq 1$, the maximizer is at $x=g(y)$. Hence,
\begin{align*}
\max_{x\leq y-y^2}f(x,y)&= \begin{cases}
\frac{(\Delta T)^2}{16T_C^2}y^2,\quad 0<y\leq y_0\\
\frac{\Delta T}{2T_C}(y-y^2)-(1-y)^2,\quad y_0\leq y \leq 1 
\end{cases}.
\end{align*}
Maximizing the expressions in the two cases over $y$ gives
\begin{align*}
 \max\left\{\left(\frac{\Delta T}{3T_C+T_H}\right)^2,\frac{(\Delta T)^2}{8T_C(T_C+T_H)}\right\}= \frac{(\Delta T)^2}{8T_C(T_C+T_H)}.
\end{align*} 
This is achieved for
\begin{equation}\label{eq:oneline}
\sigma=\frac{k_BT_c}{\sqrt{\gamma M}}\frac{2(T_h+T_c)}{(T_h+3T_c)},\;
\lambda=\frac{M}{k_BT_c}\frac{(T_h+3T_c)
(T_h-T_c)}{4(T_h+T_c)^2}.
\end{equation}

\subsection{Acknowledgments}
The research was supported in part by the NSF under grants 1807664, 1839441, 1901599, and the AFOSR under FA9550-17-1-0435.

\bibliographystyle{plain}
\bibliography{refs}

\end{document}